\documentclass[final,leqno,onefignum,onetabnum]{siamltex1213}
\usepackage{geometry}
\usepackage{amssymb}
\usepackage{amsmath}
\usepackage{graphicx}
\usepackage{subfigmat}
\usepackage{amsfonts}
\usepackage{latexsym} 
\usepackage{ wasysym}
\usepackage{newlfont}
\usepackage{leqno}
\usepackage[notref]{showkeys}
\usepackage{lineno}
\usepackage[usenames]{xcolor}

\usepackage{wrapfig,booktabs}
\usepackage{amssymb}
\usepackage{url}

\usepackage{sidecap}
\usepackage{multirow}
\usepackage[printonlyused]{acronym}
\usepackage{stmaryrd}

\newtheorem{thm}{Theorem}[section]
\newtheorem{assume}{Assumption}

\newtheorem{remark}[thm]{Remark}
\numberwithin{equation}{section}
 
\makeatother

\title{Discontinuous Galerkin  Isogeometric Analysis
         of   Elliptic Diffusion  Problems on Segmentations with Gaps
    \thanks{This work was supported by  Austrian Science Fund (FWF) under the grant NFN S117-03.}} 
\author{ Christoph Hofer\footnotemark[1]\ \footnotemark[2]
\and     Ulrich Langer\footnotemark[1]\ \footnotemark[3]
\and     Ioannis Toulopoulos\footnotemark[1]\ \footnotemark[4]
}

\begin{document}
\maketitle
\slugger{SISC}{xxxx}{xx}{x}{x--x}

\renewcommand{\thefootnote}{\fnsymbol{footnote}} 
 \footnotetext[1]{Johann Radon Institute for Computational and Applied Mathematics (RICAM),
                  Austrian Academy of Sciences (\"OAW)}
 \footnotetext[2]{christoph.hofer@ricam.oeaw.ac.at}
 \footnotetext[3]{ulrich.langer@ricam.oeaw.ac.at}  
 \footnotetext[4]{ioannis.toulopoulos@oeaw.ac.at}  
\renewcommand{\thefootnote}{\arabic{footnote}}

\begin{abstract}
        We propose a new discontinuous Galerkin Isogeometric Analysis (IgA)
technique
for the numerical solution of elliptic diffusion problems 
in computational domains decomposed
into volumetric  patches  
with non-matching interfaces. Due to an incorrect segmentation procedure, it may happen that
          the interfaces of 
          adjacent subdomains don't coincide. 
In this way,
gap regions, which are not present in the original physical domain, are created.
	  	  In this paper,
          the gap region
          is considered as a subdomain of the decomposition of the computational domain
                     and  the gap boundary
          is taken as an interface between the gap and the subdomains. 
We apply a  multi-patch           approach  and derive a 
subdomain         variational  formulation which
          includes interface continuity  conditions and 
                    is           consistent with  the original variational
          formulation of the problem. The last formulation is further modified 
          by deriving interface conditions without the presence of the solution in the gap.
          Finally, the solution of this modified problem is approximated by  a special
          discontinuous Galerkin IgA technique.           		The ideas are illustrated on
          a model                     diffusion
          problem with discontinuous  diffusion coefficients. We develop
          a rigorous theoretical framework for the proposed  method clarifying the 
          influence                    of the gap size onto the convergence rate of the method. 
	  The theoretical estimates           are supported by 
          numerical examples in two- and three-dimensional computational domains.   
         \end{abstract}

\pagestyle{myheadings}
\thispagestyle{plain}
\markboth{C. Hofer, U. Langer, I. Toulopoulos}{dGIgA  on Segmentations with Gaps}

\section{Introduction}
In the numerical solution of many  realistic problems by means of Isogeometric Analysis (IgA),
the whole computational (physical) domain $\Omega$ can often not be represented by a single volume patch
that is the image of the parameter domain by a single, smooth and regular 
B-spline or NURBS map.
In this case, it is necessary to perform a decomposition
of the computational
domain $\Omega$ into subdomains, in other words, to describe the domain $\Omega$ by
multiple patches.  Typical examples are complicated 3d domains, different diffusion coefficients,
or even different mathematical models in different parts of the domain.
Superior  B-splines (NURBS, T-spline etc)  finite dimensional spaces are used, 
in order to   construct parametrizations for these subdomains \cite{LT:Hughes_IGAbook_2009}.
 It is typical for IgA
 that the same basis functions are used to approximate the solution of the problem under consideration, 
 see \cite{LT:HughesCottrellBazilevs:2005a} and  \cite{LT:Bazilevs_IGA_ERR_ESti2006a}.    
Despite the advantages, that  B-splines (NURBS etc)   offer for the parametrization of the subdomains,
some serious difficulties can arise, especially, when the {subdomains}
topologically differ a lot  from a cube.
The segmentation procedure, that starts from the geometrical description 
of the corresponding surface patches,  can lead us to non-compatible parametrizations of the geometry, 
meaning that the parametrized interfaces of adjusting subdomains are not identical after the volume segmentation, see, e.g.,  
\cite{HLT:JuettlerKaplNguyenPanPauley:2014a,HTL:NguyenPauleyJuettler:2014a,HLT:PauleyNguyenMayerSpehWeegerJuettler:2015a}
for the discussion of isogeometric segmentation pipeline.
 In this paper, we call 
a non-watertight isogeometric segmentation
also non-matching interface parametrization. 
The result of this
  phenomenon  is the creation of
overlapping subdomains  or gap regions between adjacent  subdomains.
Here we are interested 
in the later case.   
Indeed, it is an important issue  
in IgA
to devise a stable numerical procedure that can successfully be applied to  this type of decompositions. 
The contribution of this paper aims at developing 
a multipatch discontinuous Galerkin (dG) IgA 
technique for this case. 
For simplicity, we focus on the case where  we have an initial   decomposition of $\Omega$,   which gives  non-matching
parametrizations of two subdomain interfaces,
and   a gap region, say $\Omega_g$, appears between the two adjacent subdomains. 
This means that the  domain decomposition  is  given by
 $\cal{T}_{H}(\Omega\setminus \overline{\Omega}_g):= \{\Omega_l, \Omega_r\} $,
and ${\overline \Omega} = {\overline \Omega}_l \cup {\overline \Omega}_r \cup {\overline \Omega}_g$.
The  elliptic diffusion  problem,
 that we are going to consider as model problem,
has the form: Find $u: \overline{\Omega} \rightarrow    R$ such that
\begin{align}\label{0}
-\mathrm{div}(\rho \nabla u) = f \; \text{in} \; \Omega
\quad \mbox{and}	\quad 
u  = u_D   \; \text{on} \; \partial \Omega,
\end{align}
where the diffusion coefficient  $ \rho$ 
is a  patch-wise positive constant function, 
$f$ is a given source, and $u_D$ are given Dirichlet data prescribed 
on the boundary $ \partial \Omega$ of $\Omega$.    
\\
In \cite{LT:LangerToulopoulos:2014a} and \cite{LT:LangerToulopoulos:2014b},
the second and third authors
 have   studied multipatch  dG IgA methods  
for solving model diffusion problems like (\ref{0}). 
In particular, the authors considered matching interface 
domain decompositions which are compatible
with the jump discontinuities of the coefficient $\rho$. The
 weak continuity conditions across the interfaces 
 have been imposed by introducing dG numerical fluxes, 
 see, e.g., \cite{Maximiliam_DG_DD} and \cite{Rivierebook}. 
 In this way, the  solution of the problem can 
independently be approximated in every subdomain (in the  IgA frame)  and non-matching grids can be employed. 
Here we will heavily use the results from \cite{LT:LangerToulopoulos:2014a}
in order to build up
a stable   dG IgA scheme for discretizing (\ref{0}) on 
decompositions
with non-matching interfaces 
where gaps can appear.     

\par
In the present case, we 
only deal with subdomains belonging 
to $\cal{T}_{H}(\Omega\setminus \overline{\Omega}_g)$.
Thus, we need to set up an equivalent problem 
on $\overline{\Omega} \setminus \Omega_g$. 
 We first apply a multi-patch  approach 
on 
 ${\overline \Omega}_l \cup {\overline \Omega}_r \cup {\overline \Omega}_g$,
 and  derive a variational  formulation for (\ref{0}) 
 by performing integration by parts 
 over  every subdomain separately.  
 Thereafter, under some regularity
 assumptions imposed on the weak solution of (\ref{0}),  the contributions of the volume integrals on $\Omega_g$ are 
 removed, and we construct an equivalent  variational problem on $\cal{T}_{H}(\Omega\setminus \overline{\Omega}_g)$,
 where only the normal fluxes $\nabla u|_{\Omega_g}\cdot n_{\partial \Omega_g}$ on $\partial \Omega_g$ 
appear.
 However,  the information that is provided by the original data of the problem  does not help us to 
determine the 
  fluxes $\nabla u|_{\Omega_g}\cdot n_{\partial \Omega_g}$ explicitly. 
  Thus, given the regularity of $u$ in every $\Omega_i$, the normal flux terms $\nabla u_g \cdot 
n_{\partial \Omega_g}$
are replaced by Taylor  expansions using the known values of $u$ of the neighboring 
subdomains of $\Omega_g$. In that way, we settle down 
with our variational problem, where its  solution
$u$ is defined only on the subdomains belonging to $\cal{T}_{H}(\Omega\setminus \overline{\Omega}_g)$,
 which can consequently be  approximated  by the  B-spline spaces. 
We utilize this last formulation, expressed on $\cal{T}_{H}(\Omega\setminus \overline{\Omega}_g)$,
for deriving the discrete dG IgA formulation. As mentioned above, we cannot 
produce approximations to $u_g:=u|_{\Omega_g}$, 
since the B-spline spaces are not defined on $\Omega_g$.  
The accuracy of the
discrete solution is affected by the  Taylor  expansions, 
which depend 
on the gap distance 
$d_g$, 
which characterizes the maximum distance 
of the diametrically opposite points on  $\partial\Omega_g$.  
In fact, the Taylor expansions  are playing the role of a bridge for the communication between the values of the adjacent subdomains, and  will help to build up the numerical flux in the dG IgA method
over the gap region. 
In this work, based on the results of  our recent works 
\cite{LT:LangerToulopoulos:2014a,LT:LangerToulopoulos:2014b}, 
we are aiming at 
deriving an error estimate in the classical ``broken''  dG-norm $\|.\|_{dG}$ for derivation of
 the discrete solutions from the exact solution
 in terms of the 
 mesh   size $h$ and the gap distance $d_g$. 
In particular, we will show that,  if the IgA space defined on subdomains $\Omega_i$ has the approximation power 
$h^k$ and the 
gap distance is $\mathcal{O}(h^{k+\frac{1}{2}})$ (that means that the flux approximation is of the order $\mathcal{O}(h^{k})$),
then we obtain optimal convergence rate for the error in the dG norm  $\|.\|_{dG}$.
In the special case where the gap distance is $\mathcal{O}(h)$, we obtain 
a reduced discretization error of 
$\mathcal{O}(h^{\frac{1}{2}})$.
 \par
 We 
lastly 
mention that several   techniques have recently been investigated  for coupling 
 non-matching (or non-conforming) subdomains in some weak sense. 
 In \cite{Ruess_NitcCoplPtac_2015} and \cite{Nguyen_Nitche_3Dcoupli_2014}, 
Nitsche's method have  been applied to enforce  weak coupling conditions along trimmed B-spline patches. 
 In \cite{Apostolatos_Schmidt_Wuchner_Bletzinger_IJNUMEng_2014}, the most common techniques for 
weakly imposing    the continuity on the interfaces have been  applied and tested on nonlinear 
 elasticity problems. The numerical tests have been performed on non-matching grid parametrizations.  
Furthermore, mortar
 methods have been developed in the IgA  context utilizing different B-spline degrees for the Lagrange multiplier in \cite{Brivadis_IgAMortar_2015}. The method has been used for performing numerical tests on
decompositions
with non-matching interface parametrizations. 
\par
The remainder of the paper is organized as follows. In Section 2, some notations, the weak 
form of the problem and  the definition of the  B-spline spaces are given.  
We further describe the gap region. In Section 3,  we present the problem in $\Omega\setminus\overline{\Omega}_g$, the approximation of the normal fluxes on the $\partial \Omega_g$, and the 
dG IgA scheme. In the last part of this section, we estimate the  remainder terms in the Taylor expansion. 
Section 4 is devoted to the derivation of the a priori error estimates. 
 Finally, in Section 5, we present numerical tests for validating the theoretical results 
 on two- and three- dimensional test problems. The paper closes with some conclusions  in Section 6. 
\section{Preliminaries, dG IgA notation and gap representation}
We start with some preliminary definitions and notations. 
Let $\Omega$ be a bounded Lipschitz domain in $\mathbb{R}^d$, 
and let $\alpha=(\alpha_1,...,\alpha_d)$ be a multi-index of non-negative integers 
$\alpha_1$,...,$,\alpha_d$
with degree $|\alpha| = \sum_{j=1}^d\alpha_j$,
where we are primarily interested in the cases $d=2$ and $d=3$.
For any  $\alpha$,  we define the differential operator
$D^\alpha=D_1^{\alpha_1} \cdot \cdot \cdot D_d^{\alpha_d}$,
with $D_j = \partial / \partial x_j$, $j=1,\ldots,d$, and $D^{(0,...,0)}u=u$.
For a  non-negative integer $m$, let $C^{m}(\Omega)$ denote the space
of all functions $\phi:\Omega \rightarrow \mathbb{R}$, whose 
partial derivatives
$D^\alpha \phi$ of all orders $|\alpha| \leq m$ are continuous in $\Omega$. 
We denote the subset of all functions 
from $C^\infty(\Omega)$
with compact support in $\Omega$ 
by $C_0^\infty(\Omega)\, (\text{or}\,\cal{D}(\Omega))$.
Let $1\leq p <\infty$ be fixed and $l$ be a non-negative integer.
As usual, 
$L^p(\Omega)$ denotes the Lebesgue spaces for which 
$\int_{\Omega}|u(x)|^p\,dx < \infty$, endowed with the norm
$\|u\|_{L^p(\Omega)} = \big(\int_{\Omega}|u(x)|^p\,dx\big)^{\frac{1}{p}}$, 
and $W^{l,p}(\Omega)$  is the Sobolev space, 
which consists of the functions $\phi:\Omega\rightarrow \mathbb{R}$
such that their weak derivatives $D^\alpha \phi$ with $|\alpha| \leq l$ belong to $L^p(\Omega)$.
If $\phi\in W^{l,p}(\Omega)$, then its norm is defined by 
\begin{equation*}
 \|\phi\|_{W^{l,p}(\Omega)} = \big(\sum_{0\leq |\alpha| \leq l} \|D^{\alpha}\phi\|_{L^p(\Omega)}^p\big)^{\frac{1}{p}}
 \;\; \mbox{and}\;\;
 \|\phi\|_{W^{l,\infty}(\Omega)} = \max_{0\leq |\alpha| \leq l} \|D^{\alpha}\phi\|_{L^\infty(\Omega)},
 \end{equation*}
for $1 \le p < \infty$ and $p=\infty$, respectively.
We refer  to \cite{Adams_Sobolevbook} for more details about Sobolev spaces.\\
We end this section by recalling  H\"older's and Young's inequalities: 
	for any $\epsilon$, $0<\epsilon<\infty$,  and  $1\leq p,q\leq \infty$ such that $\frac{1}{p}+\frac{1}{q}=1$, for all $u\in L^p(\Omega)$ and $v\in L^q(\Omega)$, there holds 
           \begin{align} \label{5a_01}
         \left|  \int_{\Omega}u v\,dx \right| \leq  \|u\|_{L^p(\Omega)}\|v\|_{L^q(\Omega)},\quad
         \left|  \int_{\Omega}u v\,dx \right|\leq \frac{\epsilon}{p}\|u\|^p_{L^p(\Omega)}+  \frac{\epsilon^{-\frac{q}{p}}}{q}\|v\|^q_{L^q(\Omega)}.
           \end{align}
\subsection{Weak formulation}
The weak formulation of the boundary value problem (\ref{0}) reads as follows:
for given source function $f \in L^2(\Omega)$ and Dirichlet data $u_D\in W^{\frac{1}{2},2}(\partial \Omega) $,
find a function $u\in W^{1,2}(\Omega)$ such that $u=u_D$ on $\partial \Omega$
and satisfies the variational identity
\begin{equation}
\label{4a}
a(u,\phi)=l_f(\phi),\;  \forall \phi \in 
W^{1,2}_{0}(\Omega) = \{\phi \in W^{1,2}(\Omega):\, \phi = 0\; \mbox{on}\; \partial \Omega \},
\end{equation}
where the bilinear form $a(\cdot,\cdot)$ and the linear form $l_f(\cdot)$
are defined by 
\begin{equation}
\label{4b}
a(u,\phi)=\int_{\Omega}\rho\nabla u\nabla\phi\,dx
\quad \mbox{and}\quad
l_f(\phi)=\int_{\Omega}f\phi\,dx,
\end{equation}
respectively. 
Since we assume that the diffusion coefficient $\rho \in L^\infty(\Omega)$ and uniformly positive
(later we will specify this assumption for multi-patch domains), 
the Lax-Milgram Lemma immediately yields 
existence and uniqueness of the solution $u$ of our model diffusion problem (\ref{4a}).
For simplicity, we only consider non-homogeneous Dirichlet boundary conditions on $\partial \Omega$.
However, the analysis presented in our paper can easily be generalized to 
other constellations of boundary conditions which ensure existence and uniqueness
such as Robin or mixed boundary conditions.
\par
For  developing of  the convergence analysis 
of
the  dG IgA method, that we are going to propose on subdomain decompositions with gap regions, 
we assume that the domain $\Omega\subset \mathbb{R}^d$ consists of
two subdomains $\Omega_1$ and $\Omega_2$ with common interface $F$, i.e., 
\begin{align}\label{01c}
 \overline{\Omega} = \overline{\Omega}_1\cup \overline{\Omega}_2,\quad  
   \Omega_1\cap \Omega_2 = \emptyset, \quad  \overline{\Omega}_1\cap \overline{\Omega}_2 =F.
\end{align}
For this decomposition, we use the notation $\cal{T}_{H}(\Omega)=\{\Omega_i\}_{i=1}^2$,
and  define the space
\begin{align}\label{01c_0}
	W^{l,p}(\cal{T}_{H}(\Omega))=\{u\in L^{p}(\Omega): u|_{\Omega_i}\in W^{l,p}(\Omega_i),\, \text{for}\,i=1,2\},
\end{align}
where  $l\geq 0$ is an integer  and  $1\leq p\leq \infty$ is real number.
For the solution, we consider the following  regularity assumption.
{
\begin{assume}\label{Assumption1}
We allow the diffusion coefficient $\rho$ to be positive and patch-wise constant, i.e., 
$\rho=\rho_i = \mbox{const. } > 0$ in $\Omega_i$ for $i=1,2$.
 We assume that the  solution $u$ of (\ref{4a}) belongs to  
$V= W^{1,2}(\Omega)\cap W^{l,p}(\cal{T}_H(\Omega))$ with some
 $l \geq  2$ and $p\in (\max\{1,\frac{{2d}}{(d+2(l-1))}\},2]$.
\end{assume}
}

In what follows, positive constants $c$ and $C$ appearing in  inequalities are 
generic constants which do not depend on the mesh-size $h$. In many cases,  
we will indicate on what may the constants depend for an easier understanding of the proofs. 
Frequently, we will write $a\sim b$ meaning that $c \, a\leq b \leq C \, a$. 
\subsection{Incorrect segmentation and  non-matching parametrizations}
\label{Section_2_2}
Let us suppose for the 
time being
that we have constructed the B-spline (or NURBS) parametric space (see next Subsection), 
and we start the  segmentation procedure in order to calculate the control point net for every $\Omega_i$, $i=1,2$. 
Given the control net and the B-spline space we consider    each of $\Omega_i, \, i=1,2$  as the image,  of a
B-spline parametrization  mapping. For reasons that we will see below,  we   denote
the two images  $\Omega_l$ and $\Omega_r$ correspondingly.
  The domain $\Omega$ can be consequently   described by means of  $\Omega_l$ and $\Omega_r$.  
  Despite the superior properties of
 B-spline  spaces, in several cases, as for example when $\Omega_i,\,i=1,2$ differ topologically a lot by a cube, 
 the previous geometric B-spline parametrizations
 can lead us to non-compatible parametrizations
of the common interface. 
This means that the parametrizations of the common interface of the adjusting subdomains
are not identical,  
see, e.g., \cite{HLT:JuettlerKaplNguyenPanPauley:2014a}.
We will call this situation 
a non-matching interface parametrization. 
The result of this phenomenon is the creation of overlapping subdomains or gap regions between 
 $\Omega_l$ and $\Omega_r$. Here we are interested in the later case.
For the purposes of this
paper, it suffices to consider the case where only one gap region, say $\Omega_g$, exists between
$\Omega_l$ and $\Omega_r$, and either $\Omega_l \subset \Omega_1$ or $\Omega_r \subset \Omega_2$.
As an immediate result we have that  
$\overline{\Omega} = \overline{\Omega}_l\cup \overline{\Omega}_g \cup \overline{\Omega}_r$, 
see an illustration in Fig. \ref{Sub_doms_gap}(c). 
In what follows, we will call $\Omega_l$ and $\Omega_r$ 
parametrized subdomains or simply
subdomains, if there is no chance of confusion with $\Omega_i,\,i=1,2$. 
We denote by 
$\mathcal{T}_H(\Omega\setminus\overline{\Omega}_g)= \{{\Omega}_l,{\Omega}_r\}$.
%
\subsection{B-spline spaces}
\label{Bsplinespace}
In this section, we briefly present the B-spline spaces and the form of the B-spline parametrizations
for  the physical  subdomains. 
For a more detailed presentation we refer to \cite{LT:Hughes_IGAbook_2009}, 
\cite{CarlDeBoor_Splines_2001}, \cite{LT:Shumaker_Bspline_book}. 
\par
Let us consider the unit cube $\widehat{\Omega}=(0,1)^d\subset \mathbb{R}^d$, which we will refer to as the parametric domain and
let $\overline{\Omega}=\bigcup_{i=1}^N \overline{\Omega}_i$, with ${\Omega}_i\cap{\Omega}_j=\emptyset,$ for $i\neq j$ be a decomposition
of $\Omega$. 
Let the integers $k$,  $i=1,...,N$ and $n_\iota,\,\iota=1,...,d$ denote  the
 given  B-spline degree,  the corresponding physical $i-th$ subdomain, and 
 the number of basis functions
 of the B-spline space that will be constructed in $x_\iota$-direction.  
We introduce the $d-$dimensional  vector of knots 
$\mathbf{\Xi}^d_i=(\Xi_i^1,...,\Xi_i^{\iota},...,\Xi_i^d),$\, $\iota = 1,\ldots,d$, 
 with the particular components given by 
 $\Xi_i^{\iota}=\{0=\xi^{\iota}_1 \leq \xi^{\iota}_2 \leq ...\leq \xi^\iota_{n_\iota+k+1}=1\}$. 
 The components $\Xi_i^{\iota}$  of $\mathbf{\Xi}^d_i$  form 
a mesh  $T^{(i)}_{h_i,\widehat{\Omega}}=\{\hat{E}_m\}_{m=1}^{M_i}$ in $\widehat{\Omega}$,
where $\hat{E}_m$ are the micro elements and $h_i$ is the mesh size, which is
defined as follows. Given a micro element $\hat{E}_m\in T^{(i)}_{h_i,\widehat{\Omega}} $, 
we set $h_{\hat{E}_m}=diameter(\hat{E}_m)=\max_{x_1,x_2 \in \overline{\hat{E}}_m}\|x_1-x_2\|_d $, where $\|.\|_d$
is the Euclidean norm in $\mathbb{R}^d$ and the 
subdomain mesh size $h_i$ is defined to be  $h_i= \max\{h_{\hat{E}_m}\}$.
We define $h=\max_{i=1,...,N}\{h_i\}$.

\begin{assume}\label{Assumption2}
	The  meshes $T^{(i)}_{h_i,\widehat{\Omega}}$ are quasi-uniform, i.e.,
	there exist a constant $\theta \geq 1$ such that 
	$\theta^{-1} \leq {h_{\hat{E}_m}}/{h_{\hat{E}_{m+1}}} \leq \theta$.
	Also, we assume that $h_i \sim h_j$ for $i\neq j$.
\end{assume}
\par
 Given the knot vector $\Xi_i^{\iota}$ in every direction $\iota=1,...,d$, 
 we construct the associated univariate B-spline functions, 
 $\hat{\mathbb{B}}_{\Xi_i^{\iota},k}=\{\hat{B}_{1,\iota}^{(i)}(\hat{x}_\iota),...,\hat{B}_{n_{\iota},\iota}^{(i)}(\hat{x}_\iota)\}$
 using the Cox-de Boor recursion formula, see
 details in \cite{LT:Hughes_IGAbook_2009}, \cite{CarlDeBoor_Splines_2001}.
On  the mesh $T^{(i)}_{h_i,\widehat{\Omega}},$  we define the multivariate
B-spline space, $\hat{\mathbb{B}}_{\mathbf{\Xi}^d_i,k},$
to be the tensor-product of the corresponding univariate $\mathbb{B}_{\Xi_i^{\iota},k}$ spaces. 
Accordingly,  the  B-spline functions of $\hat{\mathbb{B}}_{\mathbf{\Xi}^d_i,k}$ are defined by the tensor-product of the univariate B-spline basis functions, that is 
\begin{align}\label{0.00b1}
\hat{\mathbb{B}}_{\mathbf{\Xi}^d_i,k}=\otimes_{\iota=1}^{d}\hat{\mathbb{B}}_{\Xi_i^{\iota},k}
=span\{\hat{B}_{j}^{(i)}(\hat{x})\}_{{j}=1}^{n=n_1\times ...\times n_\iota\times ...\times n_d},
\end{align}
where each  $\hat{B}_{j}^{(i)}(\hat{x})$ has the form
\begin{align}\label{0.00b2}
\hat{B}_{j}^{(i)}(\hat{x})=&\hat{B}_{j_1}^{(i)}(\hat{x}_1)\times...
           \times \hat{B}_{j_\iota}^{(i)}(\hat{x}_\iota)\times...
           \times\hat{B}_{j_d}^{(i)}(\hat{x}_d), \,
           \text{with}\,\hat{B}_{j_\iota}^{(i)}(\hat{x}_\iota) \in \hat{\mathbb{B}}_{\Xi_i^{\iota},k}.
\end{align}
\par
 Finally, having the B-spline spaces and the B-spline control points
 $C_{j}^{(i)}$, we can represent each  subdomain $\Omega_i,\, i=1,...,N$ 
by the parametric mapping
\begin{align}\label{0.0c}
 \mathbf{\Phi}_i: \widehat{\Omega} \rightarrow \Omega_i, \quad
 x=\mathbf{\Phi}_i(\hat{x}) = \sum_{{j=1}}^nC^{(i)}_{j} \hat{B}_{j}^{(i)}(\hat{x})\in \Omega_i,
 \end{align}
\label{0.0c2}
 where $\hat{x} = \mathbf{\Psi}_i(x):=\mathbf{\Phi}^{-1}_i(x)$, cf. \cite{LT:Hughes_IGAbook_2009}.
 \par  
 We construct a mesh $T^{(i)}_{h_i,\Omega_i} =\{E_{m}\}_{m=1}^{M_i}$
 for every $\Omega_i$, whose vertices are the images of the vertices
 of the corresponding parametric mesh $T^{(i)}_{h_i,\widehat{\Omega}}$
 through $\mathbf{\Phi}_i$. 
 For each $E\in T^{(i)}_{h_i,\Omega_i}$, we denote its support extension by $D_E^{(i)}$, where the support extension
 is defined to be the interior of the set formed by the union of the supports of all B-spline functions whose supports intersects $E$. 
\par
For $i=1,...,N$, we construct the B-spline space $\mathbb{B}_{\mathbf{\Xi}^d_i,k}$ on $\Omega_i$ by
\begin{align}\label{0.0d2}
	\mathbb{B}_{\mathbf{\Xi}^d_i,k}:=\{B_{{j}}^{(i)}|_{\Omega_i}: B_{j}^{(i)}({x})=
  \hat{B}_{j}^{(i)}\circ \mathbf{\Psi}_i({x}),{\ }\text{for}{\ }
  \hat{B}_{j}^{(i)}\in \hat{\mathbb{B}}_{\mathbf{\Xi}^d_i,k} \}. 
\end{align}
The global  B-spline space $V_{h}$ with components on every $\mathbb{B}_{\mathbf{\Xi}^d_i,k}$
is defined by
\begin{align}\label{0.0d1}
V_h:=V_{h_1}\times ...\times V_{h_N}:=\mathbb{B}_{\mathbf{\Xi}^d_1,k}\times ...\times  \mathbb{B}_{\mathbf{\Xi}^d_N,k}.
\end{align}
We refer the reader to \cite{LT:Hughes_IGAbook_2009} for more information 
about the meaning of the knot vectors in CAD and IgA. 
\begin{remark}\label{remark_00}
The B-spline spaces presented above are referred to the general case of $N$ subdomains. In this paper,
 for the sake of simplicity, we assume that we have  $N=2$.
	The mappings in (\ref{0.0c}) produce (and are referred to)  matching interface parametrizations. 
	Throughout the paper it is understood that we study the case where the mappings in (\ref{0.0c}) produce
	non-matching interface parametrizations and a gap region appears between the adjacent subdomains, see Section \ref{Section_2_2}.
\end{remark}
\begin{assume}\label{Assumption2_1}
We assume   that  $k\geq l$, cf.  Assumption \ref{Assumption1}. 
\end{assume}
\begin{figure}
 \begin{subfigmatrix}{3}
 \subfigure[]{\includegraphics[width=5cm]{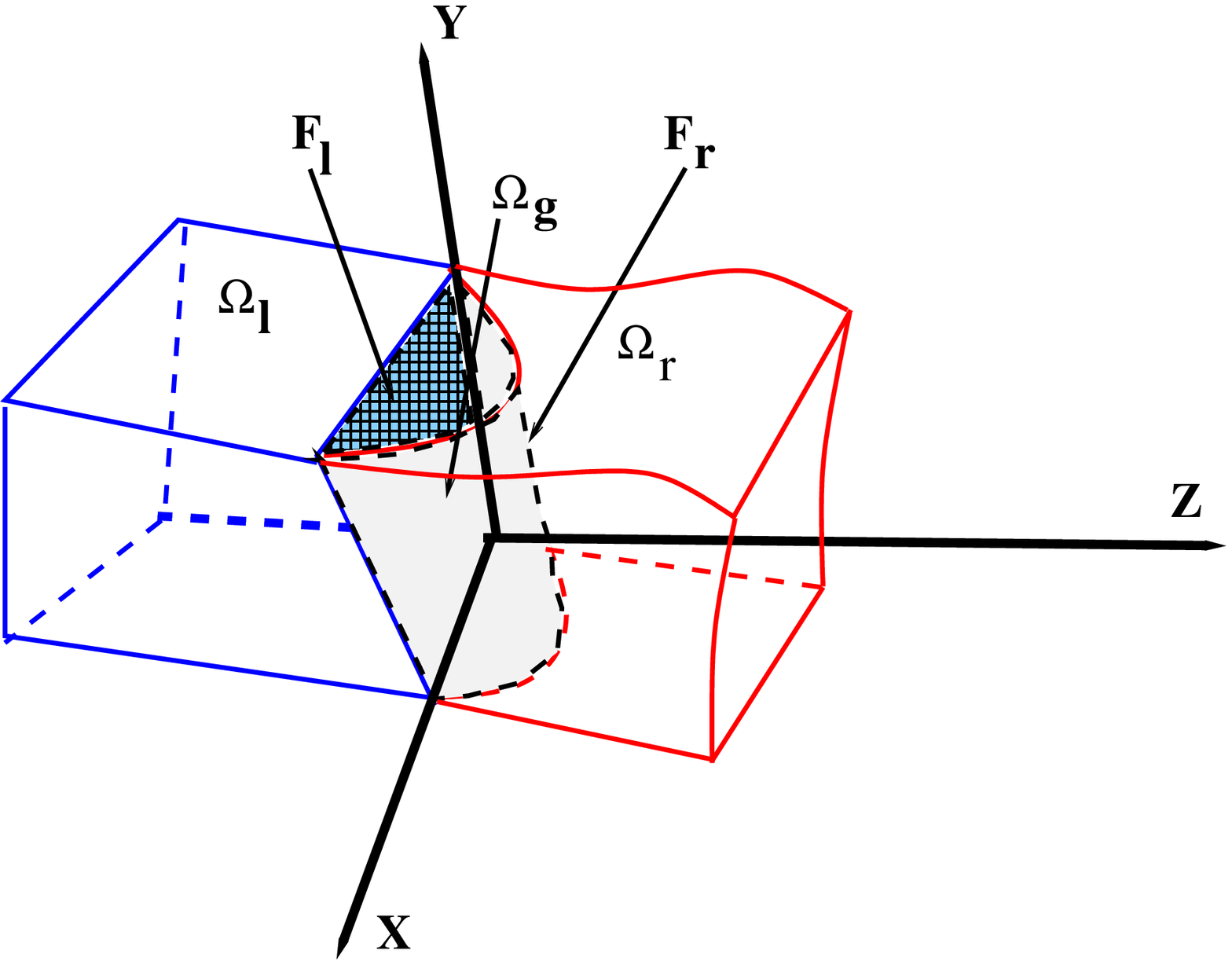}}
  \subfigure[]{\includegraphics[width=4.0cm]{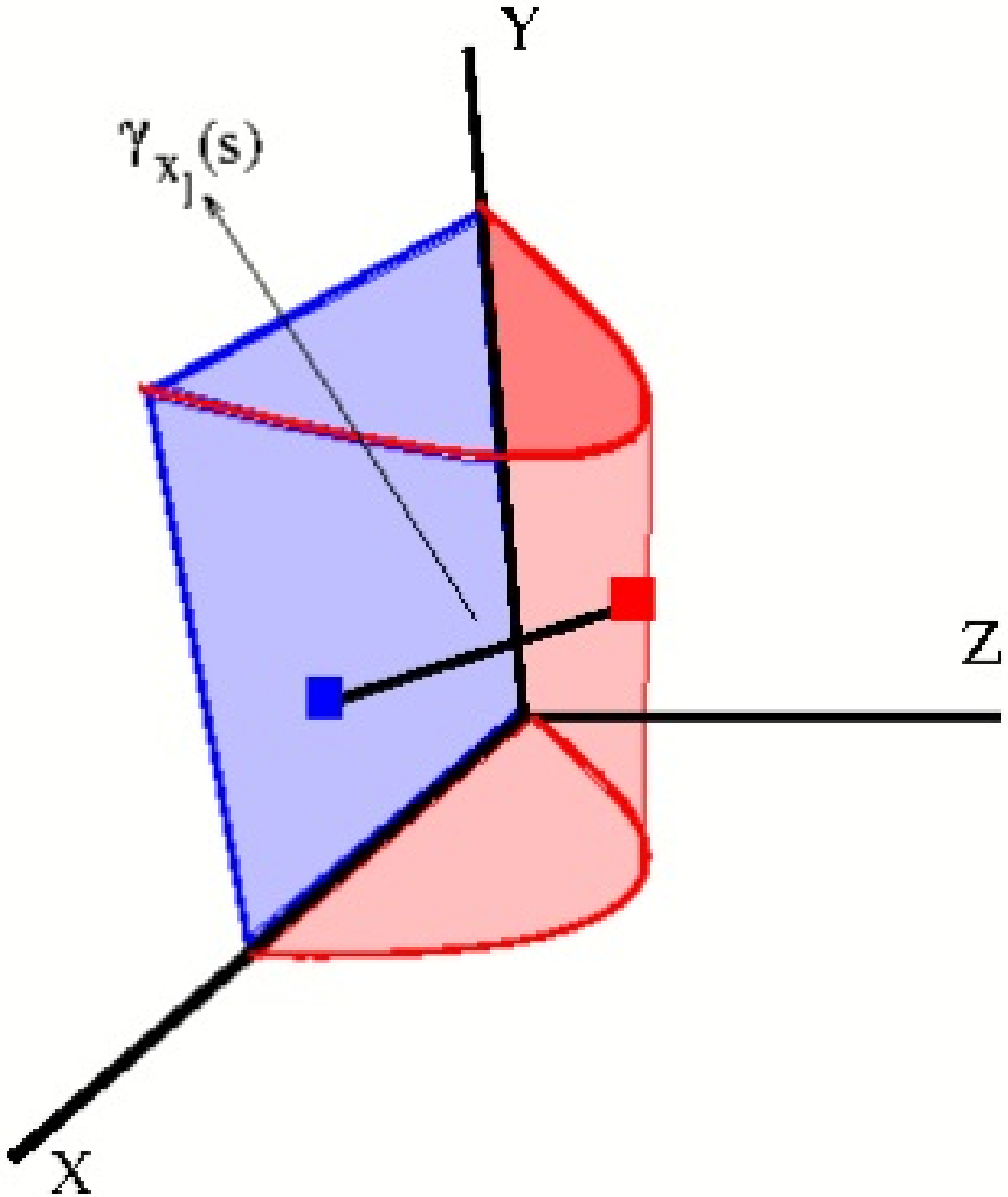}}
  \subfigure[]{\includegraphics[width=4.0cm]{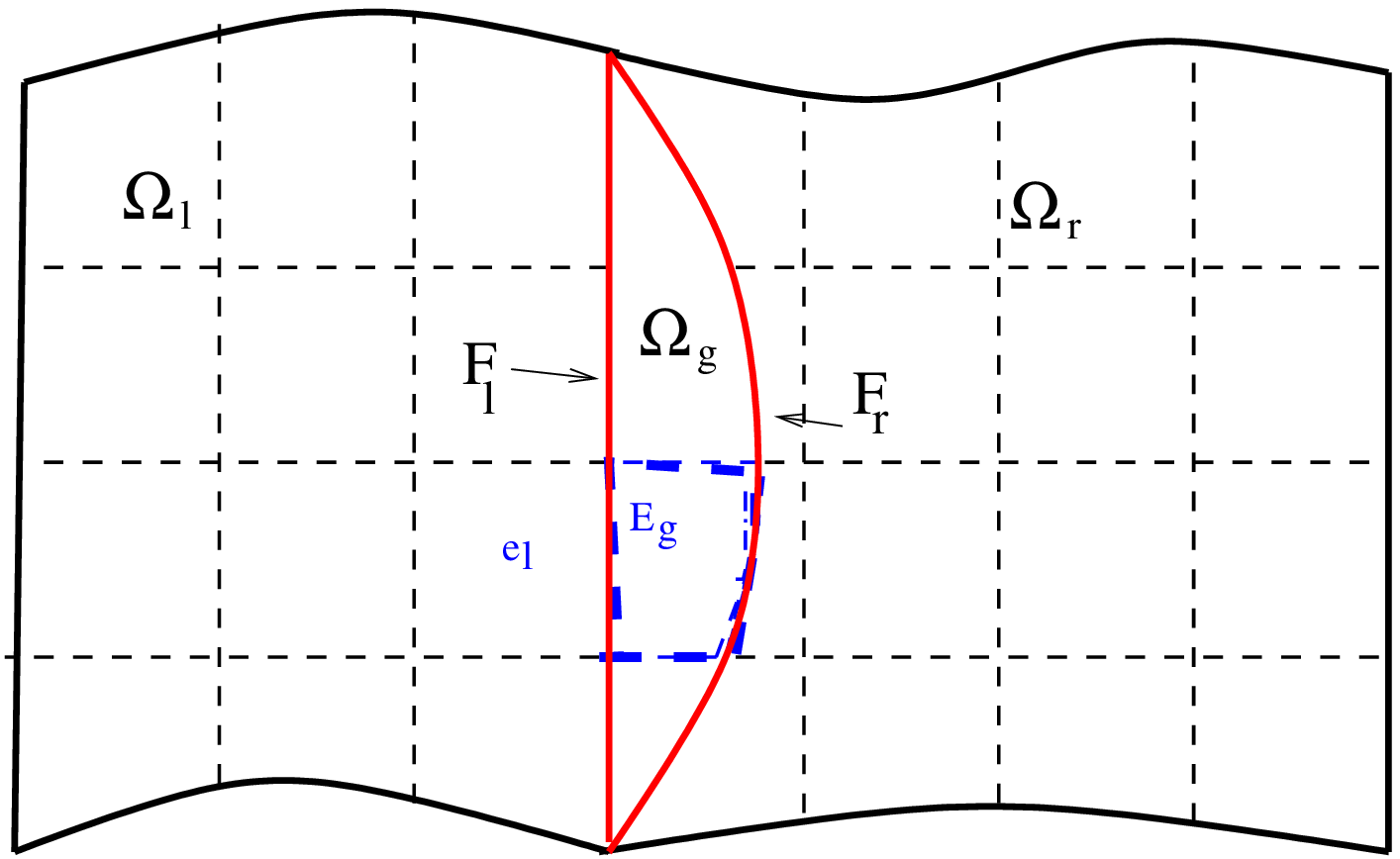}}
\end{subfigmatrix}
 \caption{(a)  Illustration of the gap region between two adjacent sub domains in $d=3$ case, (b)
 the diametrically opposite points located on $\partial \Omega_g$ in $d=3$ case, (c) an illustration for the $d=2$  case. }
 \label{Sub_doms_gap}
\end{figure}
\subsection{The  gap region}
\label{gab_region}
Let us now suppose that $\mathbf{\Phi}_l:\widehat{\Omega}\rightarrow \Omega_l$ and $\mathbf{\Phi}_r:\widehat{\Omega}\rightarrow \Omega_r$ are two parametrization mappings that give a non-matching interface parametrization.
Let  $\Omega_g$, be the gap region   
 between  $\Omega_{l}$ and $ \Omega_{r}$ and  let us consider that 
 $\partial \Omega_g=F_l\cup F_r$  with $F_l \subset \partial \Omega_l$ and 
 $ F_r\subset \partial \Omega_r$, see Figs. \ref{Sub_doms_gap}(a) and \ref{Sub_doms_gap}(c). 
We further  assume that there is a $h_0$ such that
 \begin{equation}\label{5a0_0}
\Omega_1=\Omega_l,\quad
\Omega_g \subset \Omega_2,\quad
\text{and}\quad \overline{\Omega}_2 = \overline{\Omega}_g \cup  \overline{\Omega}_r , \, \quad \forall h\leq h_0,
\end{equation}
which implies that  $F_l:=F$, see (\ref{01c}).
For simplicity,   $F_l$ is considered to be a simple region, and it can be described
 as the set of points $(x,y,z)$ satisfying
\begin{equation}\label{5a0}
0\leq y  \leq y_M,{\ } \psi_1(y)  \leq x \leq \psi_2(y),{\ }  z=0,
\end{equation}
where $y_M$ is a fixed real number and $\psi_i$  with $i=1,2$ are  given  continuous  functions.  An illustration  
 is shown in Fig. \ref{Sub_doms_gap}(b). 
Our  next goal  is to  assign   the  points  $x_l\in F_l$ to the   points of the other face $ x_r\in F_r$,
in order to build up later  the numerical flux function.  The  assignment between the opposite points is
achieved by considering $F_r$ as a graph of a B-spline function $\zeta_0(x_{l,1},x_{l,2})$. 
In particular, having the description (\ref{5a0}) for $F_l$ with unit normal vector $n_{F_l}=(0,0,1)$,  we  construct a parametrization
 $\mathbf{\Phi}_{l,r}:F_l \rightarrow F_r $ for  $F_r$ of the form
 \begin{align}\label{parmatric_lr}
 	\mathbf{\Phi}_{l,r}(x_l)=x_l+ \zeta_0(x_l)n_{F_l}:=x_r\in F_r,
 	\intertext{or more analytically  }
 	F_r=\{x_r:(x_{r,1}=x_{l,1},x_{r,2}=x_{l,2},x_{r,3}=x_{l,3}+\zeta_0(x_l))\},
 \end{align}
 where the  B-spline $\zeta_0$ 
function is determined by $\mathbf{\Phi}_r|_{F_r}$. 
We define the corresponding   mapping
$\mathbf{\Phi}_{r,l}:F_r \rightarrow F_l$ to be the projection of the $F_r$ graph onto $xy-$plane, that is 
 \begin{align}\label{parmatric_rl}
        \mathbf{\Phi}_{r,l}(x_r)= (x_{l,1},x_{l,2},0),{\ }\text{where} \quad
        \mathbf{\Phi}_{l,r}(x_{l,1},x_{l,2},0)=x_r.
 \end{align}
We will see later that the  parametrization mapping  (\ref{parmatric_rl}) simplifies  the analysis
 and is convenient for performing the numerical tests.  
 \begin{remark}\label{remark_01}
  One could consider a mapping $\mathbf{\Phi}_{r,l}:F_r \rightarrow F_l$ having the same form as the mapping
  $\mathbf{\Phi}_{l,r}$, e.g. $\mathbf{\Phi}_{r,l}=x_r+ \tilde{\zeta}_0(x_r)n_{F_r}:=x^*_l\in F_l$,
  where $n_{F_r}$ is the normal vector on $F_{r}$ inward to $\Omega_g$.
  In this case,  the point $x^*_{l}$   is different (in general) from the original point $x_l=x_{r}-\zeta_0(x_l)n_{F_l}$ used in (\ref{parmatric_lr}) and this seems to make the analysis more complicated. This is because, when
      we will construct later  the numerical flux on $F_l$ and on $F_r$, we have to take
 into account two different   assignments of the diametrically opposite points $x_l$ and $x_r$. 
 As we will see later using a parametrization mapping as in (\ref{parmatric_rl}) simplifies a lot the analysis
 and helps on an easy materialization of the method. We note further, that we have to take into account
   the  two corresponding outward normals, i.e., the $n_{F_l}$ on $F_{l}$ and $n_{F_r}$ on $F_{r}$.
The consideration  of the general (real) vector $n_{F_r}$ on  $F_r$, causes rather further
      technical difficulties  in the numerical computations.
	Since we are intending to develop a method that can be easily  materialized
	for practical applications and having in our mind a small gap regions, see few lines below (\ref{5b_1}), we suppose 
	that the angle formed between $n_{F_l}$ and $-n_{F_r}$ is almost zero, that is
	$\cos\sphericalangle({n_{F_l},-n_{F_r}})\approx 1$.
 \end{remark}
\par
We finally  characterize the points which belong in the interval $[x_l,x_r]$. 
To this end, for every
$x_l\in F_l$ we construct a $C^1$  one-to-one   map 
$\gamma_{x_l}:[0,1] \rightarrow  \overline{\Omega}_g,$ 
\begin{flalign}\label{5a1}
  \gamma_{x_l}(s)=x_l+s(x_r-x_l),\quad
 \text{with}\quad \mathbf{\Phi}_{l,r}(x_l)=x_r.
\end{flalign}
 The function $\gamma_{x_l}$ help us to quantify the size of the gap  by introducing the gap distance defined by
\begin{align}\label{5b_0}
d_g=&\max_{x_l}\{|\gamma_{x_l}(0)- \gamma_{x_l}(1)|\}.
\end{align}
Of special interest in our analysis are  gap regions whose distance decreases polynomially in $h$.
\begin{assume}\label{Assumption_dg_size}
We assume that   for any $h\leq h_0$, see (\ref{5a0_0}), holds
\begin{align}\label{5b_1}
d_g \leq &  h^\lambda,\quad \text{with}\quad  \lambda \geq  1.
\end{align}
\end{assume}
Finally, based on the previous properties of the {shape of $\Omega_g$},   
 without loss of generality, we can  assume  
that the parametrization $\mathbf{\Phi}_{l,r}$ in (\ref{parmatric_lr}) has the form
\begin{align}\label{parmatric_lr_1}
 	\mathbf{\Phi}_{l,r}(x_l)=x_l+ d_g\, \zeta(x_l)n_{F_l},\quad x_l\in F_l, 
 \end{align}
 where $\zeta$ is a B-spline and $\|\zeta\|_{L^\infty} =1$. In  Section \ref{Section_numerics} \textit{Numerical tests},
  we give explicitly the form of the mapping $ \mathbf{\Phi}_{l,r}$.

   \subsubsection{Properties of the parametrization mappings}\par
 Let us denote by $D\mathbf{\Phi}_{l,r}(x_l)$ the Jacobian matrix of $\mathbf{\Phi}_{l,r}(x_l)$ 
 evaluated at $x_l=(x_{l,1},x_{l,2},x_{l,3})\in F_l$ and
 by $D^\top\mathbf{\Phi}_{l,r}(x_l)$ its transpose. By an application of the chain rule we can verify that
 \begin{align}\label{5c_1}
  	\nabla( u\circ \mathbf{\Phi}_{l,r}(x_l) ) = D^\top\mathbf{\Phi}_{l,r}(x_l) (\nabla u)\circ \mathbf{\Phi}_{l,r}(x_l).
  \end{align}
  If $u$ is a function defined on the parametrized surface $F_{r}$ then holds
  \begin{subequations}\label{5c_2}
  \begin{align}
  	\int_{F_r}u\,d F_{r} =&  \int_{F_l}u(\mathbf{\Phi}_{l,r}(x_l)) J\,d x_l,\\
  	\int_{F_r}\nabla u \cdot n_{F_l}\,d F_{r} =&
  	\int_{F_l} D^\top\mathbf{\Phi}_{l,r}(x_l) (\nabla u)\circ \mathbf{\Phi}_{l,r}(x_l) \cdot n_{F_l} J\,d x_l,
  \end{align}
  \end{subequations}
  where $J=\sqrt{(d_g\zeta_{x_{l,1}})^2+(d_g\zeta_{x_{l,2}})^2 +1 }$ is the norm of the outward normal vector on $F_{r}$. 
\subsection{Jumps  and $\|.\|_{dG}$}
For the face $F_{i},\,i=l,r$, let ${n}_{F_i}$ be its  unit  normal vector  towards $\Omega_g$.  
For a 
smooth function  $\phi$ defined on $\Omega$, we 
define its interface averages and the jumps as  
\begin{align}\label{6}
 \llbracket \phi \rrbracket|_{F_{i}} &= \Big(\phi_i- \phi_g  \Big),&
 \{ \phi \}|_{F_{i}} &= \frac{1}{2}\Big(\phi_i+\phi_g  \Big),\\
 \nonumber
 \llbracket \rho\nabla \phi\rrbracket|_{F_{i}}\cdot n_{F_i}&=
 \Big(\rho_i\nabla \phi_i-\rho_g\nabla \phi_g \Big)\cdot n_{F_i},
 &
 \{\rho\nabla \phi\}|_{F_{i}}\cdot n_{F_i}&=
 \frac{1}{2}\Big(\rho_i\nabla \phi_{i}+\rho_g\nabla \phi_g \Big)\cdot n_{F_i}.
\end{align} 
Based on Assumption \ref{Assumption1}, we can infer that the exact solution satisfies, see \cite{LT:LangerToulopoulos:2014a},
\begin{align}\label{00d3_1}
	\llbracket u \rrbracket |_{F}=0,\quad\text{and}\quad \llbracket \rho \nabla u\rrbracket |_{F}=0.
\end{align}
To proceed  to our analysis, we need to define  the 
broken dG-norm, $\|.\|_{dG}$. For  $v\in V+V_h$, we define 
\begin{align}\label{0.0d2_0}
	\|v\|^2_{dG} &=\sum_{i=l,r}\Big(\rho_i\|\nabla v_i\|^2_{L^2(\Omega_i)} +
                \frac{\rho_i}{h}\|v_i\|^2_{L^2(\partial \Omega_i\cap \partial\Omega)} + 
                 \frac{\{\rho\}}{h}\|v_i  \|^2_{L^2(F_{i})}\Big),
\end{align}
Next,  we  show that
the values of the exact solution  $u$  on the opposite assigned points will coincide as $d_g \rightarrow 0$.
\begin{proposition}\label{Prop1}
 Let  the Assumption \ref{Assumption1} and  (\ref{5b_1}) hold true.
 For  $x_l\in F_l$, let $x_r\in F_r$ be its  corresponding assigned 
 point.  Then for $\phi \in \cal{D}(\Omega)$ we have  
 \begin{equation}\label{5c}
 \Big| \int_{F_l} (u(x_l)-u(x_r))\phi(x_l)\, dx_l\Big|  \xrightarrow{d_g \rightarrow 0} 0.
\end{equation}
\end{proposition}
\begin{proof}
By Assumption \ref{Assumption1}  it follows that
 \begin{multline}\label{5d}
 \Big|\int_{F_l} (u(x_l)-u(x_r))\phi(x_l)\,dx_l \Big| =
 \Big| \int_{F_l}\phi(x_l) \int_{0}^{1}\frac{d}{d s}u(x_l+s(x_r-x_l)) \,d s\,dx_l  \\
 \leq   \int_{F_l}\int_{0}^{1}|\phi(x_l)| |x_l-x_r| |Du(x_l+s(x_r-x_l))| \,d s\,d x_l.
\end{multline}
Since $0\leq s \leq 1$  the values of $z=x_l+s(x_r-x_l)$ are 
restricted in ${\Omega}_g \subset \Omega_2$ and the  integration
domain $F_l\times [0,1]\subseteq \Omega_2$. Henceforth, applying  (\ref{5a_01}), in (\ref{5d}) 
we get the estimate
\begin{align}\label{5d0}
\Big|\int_{F_l} (u(x_l)-u(x_r))\phi(x_l)\,dx_l \Big| \leq
	d_g \|\phi\|_{L^p(\Omega_2)} \|Du\|_{L^{p}(\Omega_2)},
\end{align}
that proves (\ref{5c}). 
\end{proof}
\section{The problem on $\Omega\setminus \overline{\Omega}_g$}
The next goal  is to derive a variational problem in $\Omega\setminus \overline{\Omega}_g$, 
such that its (unique) solution on the subdomains 
$\Omega_i, \,i=l,r$ coincides with the solution of (\ref{4a}).
 The bilinear form of this problem will be determined by taking into account the normal
 fluxes on $\partial \Omega_g$. Finally, 
the   problem will be discretized by
dG IgA methods. The main importance is to devise  appropriate   numerical fluxes,
which  will  use the diametrically opposite assigned values of $\nabla u$ and
$u$ on  $\partial \Omega_g$.
\par
Recall that  $\overline{\Omega}=\overline{\Omega}_l\cup \overline{\Omega}_g\cup\overline{\Omega}_r$.
Now, multiplying (\ref{0})  by a $\phi\in C^{\infty}_0(\Omega)$, 
integrating on every $\Omega_i$ separately, then performing integration by parts,
and finally summing over all $\Omega_i,\, i=l,g,r$, we get
 \begin{multline}\label{5a2}
  \int_{\Omega_l}\rho\nabla u\cdot \nabla \phi\,dx - 
   \int_{\partial \Omega_l \cap \partial \Omega}\rho\nabla u\cdot n_{\partial \Omega_l} \phi\,d\sigma  \\
  +\int_{\Omega_g}\rho\nabla u\cdot \nabla \phi\,dx 
  - \int_{F_l}\llbracket \rho\nabla u \phi\rrbracket \cdot n_{F_l} \,d\sigma 
  - \int_{F_r}\llbracket \rho\nabla u \phi\rrbracket \cdot n_{F_r} \,d\sigma \\
+  \int_{\Omega_r}\rho\nabla u\cdot \nabla \phi\,dx - 
   \int_{\partial \Omega_r\cap\partial \Omega}\rho\nabla u\cdot n_{\partial \Omega_r} \phi\,d\sigma 
  =
 \sum_i\int_{\Omega_i}f\phi\,dx.
 \end{multline}
Applying  the equality 
$\llbracket \rho \nabla u  \phi\rrbracket =
\{\rho\nabla u\}\llbracket  \phi \rrbracket + \llbracket\rho\nabla u \rrbracket \{ \phi\}$
in (\ref{5a2}),
and using  $\int\llbracket\rho\nabla u \rrbracket\cdot n_{\partial\Omega_g} \{ \phi\}\,d\sigma=0$, 
we obtain
\begin{multline}\label{5a3}
  \int_{\Omega_l}\rho\nabla u\cdot \nabla \phi\,dx - 
   \int_{\partial \Omega_l \cap \partial \Omega}\rho\nabla u\cdot n_{\partial \Omega_l} \phi\,d\sigma 
  - \int_{F_l} \{\rho\nabla u\}\cdot n_{F_l}\llbracket  \phi \rrbracket \,d\sigma\\
  -\int_{\Omega_g}\mathrm{div}(\rho\nabla u) \phi\,dx 
  -\int_{F_l}\{\rho\nabla u\}\cdot n_{F_l}\phi \,d\sigma-\int_{F_r}\{\rho\nabla u\}\cdot n_{F_r}\phi \,d\sigma\\
  +\int_{\Omega_r}\rho\nabla u\cdot \nabla \phi\,dx - 
   \int_{\partial \Omega_r\cap\partial \Omega}\rho\nabla u\cdot n_{\partial \Omega_r} \phi\,d\sigma 
  - \int_{F_r} \{\rho\nabla u\}\cdot n_{F_r}\llbracket  \phi \rrbracket\,d\sigma \\
  = \int_{\Omega\setminus \Omega_g}f\phi\,dx +\int_{\Omega_g}f\phi\,dx, 
 \end{multline}
 where we applied an integration by parts formula  on  $\Omega_g$ and used 
 $\int_{F_i} \rho_g \nabla u_g \cdot n_{\partial\Omega_g}  \phi\,d\sigma = 
 \frac{1}{2}\int_{F_i} (\rho_g \nabla u_g +\rho_i \nabla u_i)\cdot n_{\partial\Omega_g}\phi\,d\sigma$ 
 for $i=l,r$. 
 Finally, by  (\ref{4a}),  we have 
 \begin{multline}\label{5a4}
  \int_{\Omega_l}\rho_l\nabla u\cdot \nabla \phi\,dx - 
   \int_{\partial \Omega_l \cap \partial \Omega}\rho_l\nabla u\cdot n_{\partial \Omega_l} \phi\,d\sigma 
 -  \int_{F_l} \{\rho\nabla u\}\cdot n_{F_l}  \phi  \,d\sigma\\
  +\int_{\Omega_r}\rho_r\nabla u\cdot \nabla \phi\,dx - 
   \int_{\partial \Omega_r\cap\partial \Omega}\rho_r\nabla u\cdot n_{\partial \Omega_r} \phi\,d\sigma 
 - \int_{F_r} \{\rho\nabla u\}\cdot n_{F_r}  \phi \,d\sigma  \\
= \int_{\Omega\setminus\overline{\Omega}_g}f\phi\,dx,\quad  \text{for all}\quad  \phi\in C^{\infty}_0(\Omega).
 \end{multline}
 Evidently, solution $u$ of (\ref{4a})  satisfies   (\ref{5a4})  under  Assumption \ref{Assumption1}.
 \par
Let $u\in V= W^{1,2}(\Omega)\cap W^{2,2}(\cal{T}_H(\Omega))$. We introduce the  space
\begin{equation}\label{5a4_0}
V_{0,\partial \Omega}:=\{v\in W^{1,2}(\Omega_l) \cup W^{1,2}(\Omega_r){\ }|v=0 {\ }
\text{on}{\ } \partial \Omega_1\cap\partial \Omega \,\text{and}\,\partial \Omega_2\cap\partial \Omega\}.	
\end{equation}
 By the preceding analysis, we  settle down  the following problem: 
Find $\hat{u}\in W^{1,2}(\Omega_l\cup \Omega_r)$ and $\hat{u}:=u_D$ on $\partial \Omega$ such that
   \begin{align}\label{5a5}
   a_{\setminus \Omega_g}(\hat{u},v)=l_{f,\setminus \Omega_g}(v) + l_{\nabla u_g}(v),\quad v\in V_{0,\partial \Omega},
   \end{align}
   where the forms are defined by
     \begin{align}\label{5a6}
    a_{\setminus \Omega_g}(\hat{u},v)= &
    \int_{\Omega_l\cup \Omega_r}\rho\nabla \hat{u}\cdot \nabla v\,dx ,\\
    l_{f,\setminus \Omega_g}(v) + l_{\nabla u_g}(v) =&
     \int_{\Omega_l\cup \Omega_r} f v\,dx + \int_{\partial \Omega_g}\rho_g\nabla u_g 
    \cdot n_{\partial \Omega_g}v\,d\sigma.
   \end{align}
   Since  $f\in L^2(\Omega)$ and $\rho\nabla u_g     \cdot n_{\partial \Omega_g} \in L^2(\partial \Omega_g)$, 
  Lax-Milgram Lemma ensures that  problem (\ref{5a5})  has a unique solution.  
 Note that, for this case, the solution $u$ of (\ref{4a})  satisfies  problem (\ref{5a5}). 
 Therefore,   $\hat{u}$ coincides with $u$.
\par
   We return to the relation (\ref{5a4}), which will be the basis for the definition of the numerical scheme.
  The normal flux terms  $\nabla u_g \cdot n_{\partial \Omega_g}$, which appear  in
    (\ref{5a4}), e.g.,
 $\int_{F_l}\{\rho\nabla u\}\cdot n_{F_l}  \phi \,d\sigma=
 \int_{F_l}\frac{1}{2}(\rho_l\nabla u_l+\rho_g\nabla u_g)\cdot n_{F_l}  \phi \,d\sigma$, 
 are still unknown, in the sense that their values are not  predefined for an explicit use in the computations. 
  Next,   Taylor expansions are used for approximating
   these normal fluxes.  
\begin{remark}\label{remark_0}
To check consistency properties, we replace in (\ref{5a4}) the $\phi$  by $\phi_h\in V_h$ and integrate by parts on each $\Omega_i,i=l,r$ to get
\begin{multline}\label{5a_5}
 \sum_{i=l,r}\int_{\Omega_i}\rho_i\nabla u\cdot \nabla \phi_h\,dx - 
  \sum_{i=l,r} \int_{\partial \Omega_i \cap \partial \Omega}\rho_i\nabla u\cdot n_{\partial \Omega_i} \phi_h\,d\sigma 
 - \sum_{i=l,r} \int_{F_i} \{\rho\nabla u\}\cdot n_{F_i}  \phi_h  \,d\sigma\\
 = - \sum_{i=l,r}\int_{\Omega_i}\mathrm{div}(\rho \nabla u) \phi_h\,dx 
= \int_{\Omega\setminus\overline{\Omega}_g}f\phi_h\,dx.
  \end{multline}
  \end{remark}
\subsection{Approximating the normal fluxes $\nabla u_g\cdot n_{\partial \Omega_g}$}
 Our goal in the  present paragraph is to use Taylor theorem for constructing  approximations of 
 $\rho_g\nabla u_g\cdot n_{\partial \Omega_g}|_{\partial \Omega_g}$. The central idea is to apply the Taylor theorem 
 along the lines $\gamma_{x_l}$ (or   $\gamma_{x_r}$), see (\ref{5a1}), emanating from $x_l$ (or $x_r$) and heading in the direction
 of the diametrically opposite point 
 $x_r$ (or $x_l$ correspondingly).  In that way, we produce approximations of $\rho_g\nabla u_g\cdot n_{\partial \Omega_g}|_{\partial \Omega_g}$ using $u_l$ and $u_r$.
 
\par
We recall the following Taylor's formula with integral remainder,  for $f\in C^m([0,1])$
\begin{flalign}\label{7}
 f(1)=f(0)+\sum_{j=1}^{m-1}\frac{1}{j!}f^{(j)}(0)+\frac{1}{(m-1)!}\int_0^1s^{m-1}f^{(m)}(1-s)\,d s.
\end{flalign}
Let us suppose for the moment that $u\in C^m(\Omega)$. 
As usual, let $x_l=(x_{l,1},x_{l,2},x_{l,3})$ be a fixed point on $ F_l$ and $x_r=\mathbf{\Phi}_{l,r}(x_l)$.
We define $f(s)=u(\gamma_{x_l}(s))=u(x_l+s(x_r-x_l))$. By chain rule we can obtain
\begin{flalign}\label{7_a}
 f^{(j)}(s)=\sum_{|\alpha|=j}\frac{j!}{\alpha !}D^\alpha  u(x_l +s(x_r-x_l))(x_r-x_l)^\alpha,
\end{flalign}
where $\alpha!=\alpha_1!...\alpha_d!$ and $ (x_r-x_l)^\alpha=(x_{r,1}-x_{l,1})^{\alpha_1}...
(x_{r,d}-x_{l,d})^{\alpha_d}$.
Combining (\ref{7}) and (\ref{7_a}), we obtain
\begin{subequations}\label{7_b}
\begin{align}\label{7_b_1}
 u(x_r)=u(x_l) +\sum_{0<|\alpha|<m}\frac{1}{\alpha !}D^{\alpha}u(x_l)(x_r-x_l)^{\alpha} 
 +   R^mu(x_l),
\end{align}
where $R^mu(x_l)$ is the $m-th$ order remainder term defined by
\begin{align}\label{7_b_2}
    R^mu(x_l)=   \sum_{|\alpha|=m}(x_r-x_l)^\alpha\frac{m}{\alpha !}
        \int_{0}^1s^{m-1}D^\alpha  u(x_r +s(x_l-x_r))\,d s.
\end{align}
\end{subequations}
Setting  $m=2$ in (\ref{7_b}), we get
\begin{flalign}\label{7_c}
 u(x_r)=&u(x_l) +\nabla u(x_l)\cdot(x_r-x_l) + R^2u(x_l).%
\end{flalign}
Now, we use  (\ref{7_c})  to approximate the flux terms $\nabla u_g\cdot n_{F_l}$ 
 in (\ref{5a4}). 
Denoting $r_l = x_r-x_l$ and $ r_r=-r_l$, by
(\ref{parmatric_lr_1})  we conclude that $ n_{F_l}=\frac{r_l}{|r_l|}$ and $ n_{F_r}=\frac{r_r}{|r_r|}$. 
 Using that  
$0=\llbracket u\rrbracket|_{F_l}=(u_l(x_l)-u_g(x_l))$ and (\ref{7_c}), we have
\begin{subequations}\label{7_0c}
\begin{align}
	 u_r(x_r)=&u_g(x_l) +\nabla u_g(x_l)\cdot r_l  + R^2u_g(x_l) \\
	 u_g(x_l)=&u_r(x_r) -\nabla u_r(x_r)\cdot r_l  + R^2u_r(x_r),
\end{align}
\end{subequations}
and   we can find  that 
\begin{subequations}\label{7_1c}
\begin{align}
\label{7_1c_a}
	\nabla u_g \cdot n_{F_l} = &\nabla u_r \cdot n_{F_l}  -\frac{1}{|r_l|}\big(  R^2 u_r(x_r)+ R^2u_g(x_l) \big ) \\
\label{7_1c_b}
	-\frac{1}{h} \big( u_l(x_l)-u_r(x_r)\big) = & \frac{|r_l|}{h }\nabla u_g(x_l)\cdot n_{F_l}  + \frac{1}{h} R^2u_g(x_l).
\end{align}
\end{subequations}
 Next, we adopt  the notations 
 \begin{eqnarray*}
    u_l:=u_l(x_l),{\ }u_r&:=&u_r(x_r),\, x_r=\mathbf{\Phi}_{l,r}(x_l)), \\
   R^2u_g(x_r+s(x_l-x_r))      &:=&\sum_{|\alpha|=2}(x_r-x_l)^\alpha\frac{2}{\alpha !}\int_{0}^1 s D^\alpha  u(x_r +s(x_l-x_r))\,d\sigma,  \\
   R^2u_r(x_l+s(x_r-x_l)) &:=&\sum_{|\alpha|=2}(x_l-x_r)^\alpha\frac{2}{\alpha !}\int_{0}^1 s D^\alpha  u(x_l +s(x_r-x_l))\,d\sigma. 
 \end{eqnarray*}

For $\phi_h \in V_h$, it follows by (\ref{7_1c}) that 
\begin{multline}\label{7_d}
  \int_{F_l}\big(\frac{\rho_l}{2}\nabla u_l +\frac{\rho_g}{2}\nabla u_g\big)\cdot n_{F_l}\phi_h - 
  \frac{\{\rho\}}{h}\llbracket u\rrbracket \phi_h\,d\sigma = \\
   \int_{F_l}\frac{\rho_l}{2}\nabla u_l\cdot n_{F_l} \phi_h+\frac{\rho_g}{2}\nabla u_r\cdot n_{F_l} \phi_h
  -\big( \frac{\rho_g}{2|r_l|}R^2u_g(x_r+s(x_l-x_r))+ \frac{\rho_g}{2|r_l|} R^2u_r(x_l+s(x_r-x_l)) \big)\phi_h \\
  -\frac{ \{\rho\}  }{h}\big(u_l-u_r\big)\phi_h + \frac{ \{\rho\}}{h}\big(|r_l|\nabla u_g\cdot n_{F_l} + R^2u_g(x_r+s(x_l-x_r))\big)\phi_h\,d\sigma = \\    
   \int_{F_l}\Big(\frac{\rho_l}{2}\nabla u_l+\frac{\rho_g}{2}\nabla u_r\Big)\cdot n_{F_l}\phi_h  -
    \frac{\{\rho \}}{h}\big(u_l-u_r\big)\phi_h \,d\sigma - \\
   \int_{F_l}\Big( \frac{\rho_g}{2|r_l|}R^2u_g(x_r+s(x_l-x_r))+ \frac{\rho_g}{2|r_l|} R^2u_r(x_l+s(x_r-x_l)) \Big)\phi_h\,d\sigma + \\
   \int_{F_l}\frac{ \{\rho\}}{h}\Big(|r_l|\nabla u_g\cdot n_{F_l} + R^2u_g(x_r+s(x_l-x_r))\Big)\phi_h\,d\sigma. 
\end{multline}

Using that $\llbracket \rho \nabla u\rrbracket|_{F_i}=0$ for $ i=l,r$, 
the assumption that
$\cos\sphericalangle({n_{F_l},-n_{F_r}})\approx 1$, relations (\ref{5c_2}), definition (\ref{parmatric_rl})
 and relations (\ref{7_0c}) and (\ref{7_1c}), we can derive the corresponding 
form for the second flux term on $F_r$
\begin{multline}\label{7_d0}
 \int_{F_r}\Big(\frac{\rho_r}{2}\nabla u_r +\frac{\rho_g}{2}\nabla u_g\Big)\cdot n_{F_r}\phi_h
 -\frac{ \{\rho\}} {h}\llbracket u\rrbracket \phi_h\,d\sigma = \\
 \int_{F_r}\Big(\frac{\rho_r}{2}\nabla u_r+\frac{\rho_l}{2}\nabla u_l\Big)\cdot n_{F_r}\phi_h  -
    \frac{\{\rho\}}{h}\big(u_r-u_l\big)\phi_h \,d\sigma - \\
   \int_{F_r}\Big( \frac{\rho_g}{2|r_r|}R^2u_g(x_l+s(x_r-x_l))+ \frac{\rho_g}{2|r_r|} R^2u_l(x_r+s(x_l-x_r)) \Big)\phi_h\,d\sigma + \\
   \int_{F_r}\frac{ \{\rho\}}{h}\Big(|r_r|\nabla u_g\cdot n_{F_r} + R^2u_g(x_l+s(x_r-x_l))\Big)\phi_h\,d\sigma.
 \end {multline}
 
  \begin{remark}\label{remark_1}
	We   point out that  in  general holds  $n_{F_r}|_{F_r}\neq -n_{F_l}$ and based on 
	(\ref{5c_1}) and (\ref{5c_2}), we get
	\begin{align*}\int_{F_l}\rho_r D^{-\top}\nabla u_r(\mathbf{\Phi}_{l,r}(x_l))\cdot (-n_{F_l}) J\,d x_l \neq 
	\int_{F_l}\rho_r D^{-\top}\nabla u_r(\mathbf{\Phi}_{l,r}(x_l))  \cdot n_{F_r} J \,d x_l \\
=	\int_{F_r}\rho_r \nabla u_r(x_r) \cdot n_{F_r} \,d x_r,
	\end{align*}
	where $J$ is the norm of the outward normal vector on the image of  $\mathbf{\Phi}_{l,r}$. 
	However, for this general case, we have the following estimate for the fluxes in the  directions
	$-n_{F_l}$ and $n_{F_r}$. 
\end{remark} 
  
  \begin{proposition}\label{Prop1_0}
  	Let the assumptions (\ref{5a0}), (\ref{5b_1}) and (\ref{parmatric_lr_1}) concerning the 
  	shape of $\Omega_g$  and the parametrization of $F_r$ hold. Then there exist  positive constant
  	 $C_1=C(\|\zeta\|_{W^{1,\infty}})$,
  	 such that 
  	\begin{align}\label{7_d0_a1}
 \Big|\int_{F_r}\rho_r\nabla u_r\cdot n_{F_r}\,d x_r -\int_{F_r}\rho_r\nabla u_r\cdot (- n_{F_l})\,d x_r \Big| \leq &
C_{1} d_g \|\rho_r\nabla u_r \|_{L^p(F_r)}.
 \end{align}
  \end{proposition}
\begin{proof}
	Let us denote $\zeta_0(x_l)=d_g \zeta(x_l)$. It follows from the form of the parametrization $\mathbf{\Phi}_{l,r}$ that
	$\big|n_{F_r}-(-n_{F_l})\big|=\frac{1}{J}\big|\big(\zeta_{0_{x_{l,1}}},\zeta_{0_{x_{l,2}}},1-J\big)\big|$, 
	where $J=\sqrt{\zeta_{0_{x_{l,1}}}^2 +\zeta_{0_{x_{l,2}}}^2+1}$ 
	is the norm of the outward normal vector on the image of $\mathbf{\Phi}_{l,r}$. Since $1\leq J$, we can show that
	$(1-J)^2 \leq \zeta_{x_{0_{l,1}}}^2 +\zeta_{0_{x_{l,2}}}^2$, and then it follows that
\begin{align}\label{7_d0_a2}
	\big|n_{F_r}-(-n_{F_l})\big| \leq 	\sqrt{\zeta_{0_{x_{l,1}}}^2 +\zeta_{0_{x_{l,2}}}^2 +(1-J)^2}\leq
	 \sqrt{2}d_g\|\zeta\|_{W^{1,\infty}}.
	\end{align}
Now, applying inequality (\ref{5a_01}) on the left hand side of (\ref{7_d0_a1}) and using (\ref{7_d0_a2}), the desired
result  easily follows. 
\end{proof}%
\begin{proposition}\label{Prop2_0}
 Let the points $x_{l_0}$ on $F_l$ and the corresponding $x_{r_0}=\mathbf{\Phi}_{l,r}(x_{l_0})$ such that 
 $|x_{l_0}-x_{r_0}|= d_g$. Then for any $x_{l}\in F_l$ and $x_{r}=\mathbf{\Phi}_{l,r}(x_{l})$,
 see Fig. \ref{fig_rl0_rl}, there  is a constant $C=C(|F_l|,|F_r|)$ such that
 \begin{align}\label{7_d0_a}
  \frac{|x_{l_0}-x_{r_0}|}{|x_{l}-x_{r}|}=C.
 \end{align}
 \end{proposition}
 \begin{proof}
 An application of Thale's theorem on the triangle $Ox_{r_0} x_{l_0} $ gives 
 $
  	\frac{|O -x_{r}|}{|O -x_{r_0}|}=\frac{|O - x_{l}|}{|O - x_{l_0}|},$ \text{and}
  $	\frac{|x_{l_0} - Z|}{|x_{l_0} -x_{r_0}|}=\frac{|O - x_{r}|}{|O - x_{r_0}|}$. Replacing $|x_{l_0} - Z|=|x_l-x_r|$ and 
  $|x_{l_0} - x_{r_0}|=d_g$ into the last relations, the result (\ref{7_d0_a}) follows.
 \end{proof}
 
   \begin{SCfigure}
    {\includegraphics[width=3.78cm]{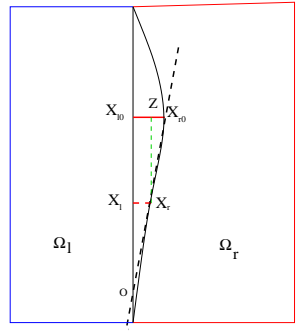}}
  \caption{Illustration of the ratio between the gap distance and the distance
   of two diametrically   opposite assigned  points.}
    \label{fig_rl0_rl}
   \end{SCfigure}
\subsection{The dG IgA   problem on $\Omega\setminus \overline{\Omega}_g$}
For  convenience we introduce the notation 
$R_{\nabla,i}=\{\rho\}\big(\frac{|r_i|}{h}\nabla u_g\cdot n_{F_i}  +\frac{1}{h}R^2u_g(x_i)\big)$ for $ i=l,r$.
Recalling (\ref{5a4}), the identity (\ref{5a_5}) and utilizing the flux approximations (\ref{7_d}) and (\ref{7_d0}), 
we  deduce that the exact solution $u$ satisfies 
\begin{multline}\label{7_d3}
 \int_{\Omega_l}\rho_l\nabla u\cdot \nabla \phi_h\,dx - 
   \int_{\partial \Omega_l \cap \partial \Omega}\rho_l\nabla u\cdot n_{\partial \Omega_l} \phi_h\,d\sigma \\
 - \int_{F_l}\Big(\frac{\rho_l}{2}\nabla u_l+\frac{\rho_r}{2}\nabla u_r\Big)\cdot n_{F_l}\phi_h  -
  \frac{\{\rho\}}{h}\big(u_l-u_r\big)\phi_h\, d\sigma \\
+\int_{F_l}\Big\{R_{\nabla,l}+\frac{\rho_g}{2|r_l|}R^2u_r(x_r)
  +\frac{\rho_g}{2|r_l|}R^2u_g(x_l)\Big\}\phi_h \,d\sigma\\
  +\int_{\Omega_r}\rho_r\nabla u\cdot \nabla \phi_h\,dx - 
   \int_{\partial \Omega_r\cap\partial \Omega}\rho_r\nabla u\cdot n_{\partial \Omega_r} \phi_h\,d\sigma \\
-   \int_{F_r}\Big(\frac{\rho_r}{2}\nabla u_r+\frac{\rho_l}{2}\nabla u_l\Big)\cdot n_{F_r}\phi_h  -
  \frac{ \{\rho\}}{h}\big(u_r-u_l\big)\phi_h \,d\sigma\\
+\int_{F_r}\Big\{R_{\nabla,r}+\frac{\rho_g}{2|r_r|}R^2u_l(x_l)
   +\frac{\rho_g}{2|r_r|}R^2u_g(x_r)\Big\}\phi_h \,d\sigma
     \\
 =  \int_{\Omega\setminus\overline{\Omega}_g}f\phi_h\,dx, {\ }\text{for}{\ }\phi_h \in V_h,
\end{multline}
where the notation for the Taylor remainders is the same as in previous paragraph.
 We observe that the terms appearing in (\ref{7_d3})  are the terms that are expected to be appear  in a
 dG scheme, of course, excluding the Taylor remainder terms. 
In view of this, we  
  define the  forms
$ B_{\setminus \Omega_g}(\cdot,\cdot):(V+V_h)\times V_h \rightarrow \mathbb{R}$, 
$R_{\Omega_g}(\cdot,\cdot):(V+V_h)\times V_h  \rightarrow \mathbb{R}$ and the linear functional
$l_{f,\setminus \Omega_g}:V_h \rightarrow \mathbb{R}$ by
 \begin{subequations}\label{7_d3_b}
\begin{align}
\nonumber
B_{\setminus \Omega_g}(u,\phi_h)= &\int_{\Omega_l}\rho_l\nabla u\cdot \nabla \phi_h\,dx - 
   \int_{\partial \Omega_l \cap \partial \Omega}\rho_l\nabla u\cdot n_{\partial \Omega_l} \phi_h\,d\sigma \\
   \nonumber
 -& \int_{F_l}\Big(\frac{\rho_l}{2}\nabla u_l+\frac{\rho_r}{2}\nabla u_r\Big)\cdot n_{F_l}\phi_h  -
  \frac{\{\rho \} }{h}\big(u_l-u_r\big)\phi_h\, d\sigma \\
  \nonumber
 +& \int_{\Omega_r}\rho_r\nabla u\cdot \nabla \phi_h\,dx - 
   \int_{\partial \Omega_r\cap\partial \Omega}\rho_r\nabla u\cdot n_{\partial \Omega_r} \phi_h\,d\sigma \\
-  \int_{F_r}&\Big(\frac{\rho_r}{2}\nabla u_r+\frac{\rho_l}{2}\nabla u_l\Big)\cdot n_{F_r}\phi_h  -
  \frac{\{\rho\}  }{h}\big(u_r-u_l\big)\phi_h \,d\sigma,\\
  \nonumber
 R_{\Omega_g}(u,\phi_h)= & \int_{F_l}\Big\{R_{\nabla,l}\phi_h+\frac{\rho_g}{2|r_l|}R^2u_r(x_r)\phi_h
  +\frac{\rho_g}{2|r_l|}R^2u_g(x_l)\phi_h\Big\} \,d\sigma \\
 +&  \int_{F_r}\Big\{R_{\nabla,r}\phi_h+\frac{\rho_g}{2|r_r|}R^2u_l(x_l)\phi_h
  +\frac{\rho_g}{2|r_r|}R^2u_g(x_r)\phi_h\Big\} \,d\sigma,\\
 l_{f,\setminus \Omega_g}(\phi_h) = & \int_{\Omega\setminus\overline{\Omega}_g}f\phi_h\,dx.
\end{align}
\end{subequations}
We note that,  the remainder  integral terms $R_{\Omega_g}$ should appear in (\ref{7_d3}).
For establishing the  dG IgA discrete problem, we 
prefer the absence of these terms in the   discrete form. Also, we wish the weak
enforcement of the Dirichlet boundary conditions. 
Thus, for defining the dG IgA scheme, we use the forms in (\ref{7_d3_b}) and  introduce the bilinear form 
$B_h(\cdot,\cdot):V_h\times V_h \rightarrow \mathbb{R}$
 and the linear form $F_h: V_h \rightarrow \mathbb{R}$
as follows
\begin{equation}\label{7_d5}
	B_h(u_h,\phi_h) = B_{\setminus \Omega_g}(u_h,\phi_h) +
	\sum_{i=l,r}\frac{\rho_i}{h}\int_{\partial\Omega_i \cap \partial \Omega}u_h \phi_h\,d\sigma,
\end{equation}
\begin{equation}\label{7_d6}
	F_h(\phi_h) = l_{f,\setminus \Omega_g}(\phi_h)+
	\sum_{i=l,r}\frac{\rho_i}{h}\int_{\partial\Omega_i\cap \partial \Omega}u_D \phi_h\,d\sigma.
\end{equation}
We consider the discrete problem: Find $u_h\in V_h $ such that
\begin{equation}\label{7_d4}
	B_h(u_h,\phi_h) = F_h(\phi_h), \quad \text{for all} {\ }\phi_h\in V_h.
\end{equation}
An immediate  result  is that,  for the exact solution  $u \in V$, the variational identity 
\begin{align}\label{7_d3_a}
 B(u,\phi_h):= B_{h}(u,\phi_h) + R_{\Omega_g}(u,\phi_h)=
  F_h(\phi_h),{\ }\forall \phi \in V_{h},
  \end{align}
  holds.
Next we show several results that are going to be used in the error analysis.
\begin{lemma}\label{lemma0_0} Let $\frac{1}{q}=\frac{p-1}{p}$ and  $\gamma_{p,d} = \frac{1}{2}d(p-2)$. 
Then there exist a constant $C\geq 0$ independent of $h$ such that
the estimate  
\begin{align}\label{7_d7a_0_1}
\frac{1}{ h^{\frac{1+\gamma_{p,d}}{p}}}\|\phi_h \|_{L^q(F_i)} \leq C h^{-\frac{1}{2}}\|\phi_h\|_{L^2(F_i)},\,{\ }\text{for}{\ }\,i=l,r,
\end{align}
holds for every $\phi_h\in V_h$.  
\end{lemma}
\begin{proof}
	The lemma is proven in \cite{LT:LangerToulopoulos:2014a}. 
\end{proof}
  
\begin{lemma}\label{lemma00} Let $\gamma_{p,d} = \frac{1}{2}d(p-2)$. 
Then there is a constant $C\geq 0$ independent of $h$ such that
the estimate 
\begin{align}\label{7_d7a}
\left| B_h(u,\phi_h) \right|
\leq & C \big(\|u\|_{dG}^p + \sum_{i=l,r}
   h^{1+\gamma_{p,d}}\| \nabla u_i\|^p_{L^p(\partial \Omega_{i})}\big)^{\frac{1}{p}}\|\phi_h\|_{dG},
\end{align}
 holds  for all $(u,\phi_h)\in (V+V_h)\times V_h$.
\end{lemma}
\begin{proof}
We first give  a bound for the normal flux terms on  $ \partial \Omega_l$. 
A direct application of   Lemma 5.2 in \cite{LT:LangerToulopoulos:2014a} gives
\begin{align}\label{7_d7a_0}
   \left|\int_{\partial \Omega_l \cap \partial \Omega}\rho_l\nabla u\cdot n_{\partial \Omega_l} \phi_h\,d\sigma \right|
   \leq  C\Big(    h^{1+\gamma_{p,d}}\| \nabla u_l\|^p_{L^p(\partial \Omega_l \cap \partial \Omega)}\Big)^{\frac{1}{p}} \|\phi_h\|_{dG}.
\end{align}
Let $J$ be the norm of the outward normal vector on the image of	 $\mathbf{\Phi}_{l,r}$. 
For the  flux terms on $F_l$, the  triangle and  (\ref{5a_01}) inequalities yield
 \begin{multline}\label{7_d7b}
 \left|\int_{F_l}\Big(\frac{\rho_l}{2}\nabla u_l+\frac{\rho_r}{2}\nabla u_r \Big)\cdot n_{F_l}\phi_h\,d\sigma \right|\leq  
   \Big|\int_{F_l} (\rho_l h^{1+\gamma_{p,d}})^{\frac{1}{p}}\nabla u_l\cdot n_{F_l}
  \frac{\rho_{l}^{\frac{1}{q}}}{h^{\frac{1+\gamma_{p,d}}{p}}} \phi_h \,d\sigma \Big| \\
+  \Big|\int_{F_l} (\rho_r h^{1+\gamma_{p,d}})^{\frac{1}{p}}\nabla u_r\cdot n_{F_l} 
J^{-1} J\frac{\rho_r^{\frac{1}{q}}}{ h^{\frac{1+\gamma_{p,d}}{p}}}\phi_h\,d\sigma  \Big| \\
\leq C_1(\rho_l)   h^{\frac{1+\gamma_{p,d}}{p}}\| \nabla u_l\|_{L^p(F_l)} 
\frac{1}{ h^{\frac{1+\gamma_{p,d}}{p}}}\|\phi_h \|_{L^q(F_l)} +
C_2(\rho,J^{-1})   h^{\frac{1+\gamma_{p,d}}{p}}\| \nabla u_r\|_{L^p(F_r)} 
\frac{1}{ h^{\frac{1+\gamma_{p,d}}{p}}}\|\phi_h \|_{L^q(F_l)}\\
\leq   C_{3}(\rho,J^{-1})   h^{\frac{1+\gamma_{p,d}}{p}}\Big(\| \nabla u_l\|_{L^p(F_l)}+ \| \nabla u_r\|_{L^p(F_r)}\Big)\|\phi_h\|_{dG},
 \end{multline}
 where the estimate  (\ref{7_d7a_0_1}) and relations (\ref{5c_2})  have been used.  
  The  flux terms of $B_h(\cdot,\cdot)$, which appear on $F_r$ can be bound in a similar way. 
 As a last step, we need to bound the jump terms   in $B_h(\cdot,\cdot)$.
 Following similar procedure as in (\ref{7_d7b}),  we can show
 \begin{align}\label{7_d7e}
 \left|\sum_{i=l,r} \int_{\Omega_i}\rho_i\nabla u\cdot \nabla \phi_h\,dx  
+   \int_{F_i}  \frac{\{\rho\}}{h}\big(u_i-u_j\big)\phi_h \,d\sigma\right|
      \leq C\|u\|_{dG} \|\phi_h\|_{dG},\,\text{for}\, j=r,l\,\text{and}\, j\neq i.
 \end{align}
 Finally, collecting all the above  bounds we can deduce assertion (\ref{7_d7a}). 
\end{proof}

Now, we  prove that the discrete problem (\ref{7_d4}) has unique solution.
\begin{lemma}\label{lemma_0}
	The bilinear form $B_h(\cdot,\cdot)$ in (\ref{7_d5}) is bounded and elliptic on $V_h$, i.e., 
	 there are positive constants $C_M$ and $C_m$ such that the estimates
	\begin{align}
		\left|B_h(v_h,\phi_h)\right| \leq C_M \|v_h\|_{dG}\|\phi_h\|_{dG}
		\quad \text{and}\quad  B_h(v_h,v_h) \geq C_m \|v_h\|^2_{dG},                             
	\end{align}
	hold for all $\phi_h\in V_h$.
\end{lemma}
\begin{proof}
	The two properties of $B_h(\cdot,\cdot)$ can be shown following the same procedure as in Lemma \ref{lemma00}
	and mimic the proofs  of Lemma 4.5
	and Lemma 4.6 	in \cite{LT:LangerToulopoulos:2014a}. Thus, the details are  omitted. 
\end{proof}
\par
Since $B_h(.,.)$ is bounded and elliptic in $V_h$,
 we can  apply the Lax-Milgram 
theorem to conclude that the  problem (\ref{7_d4})  has a unique solution. 
\par
One of the most important properties of the dG discretization is its consistency. This
ensures that the ``right'' equations are solved. Consistency yields 
  Galerkin orthogonality.  
 Here, the solution $u$ satisfies   (\ref{7_d3_a}) but 
does not  satisfy the discrete problem  (\ref{7_d4}).
We derive the error analysis borrowing  ideas from the weak consistent FE methods, \cite{ERN_FEM_book}. 
We start with the derivation of  uniform bounds for the $R_{\Omega_g}(u,\phi_h)$ terms. 
\subsection{ Estimates of the remainder terms}
We proceed  by deriving   estimates for general order  Taylor remainder terms, see (\ref{7_c}). 
We have the following estimate.
\begin{lemma}\label{lemma_01}
 For $(u,\phi_h) \in V\times V_h$, and for $i=l,r$ there is a positive constant $C$, such that   
   \begin{align}\label{7_00_e}
  \frac{|r_i|}{h}\big |\int_{F_i}\rho_g\nabla u_g\cdot n_{F_i}\phi_h\,d \sigma \big | \leq C h^{\lambda-1}h^{\frac{1+\gamma_{p,d}}{p}}
  \|\nabla u_g\|_{L^p(F_i)}  \|\phi_h\|_{dG},
  \end{align}
  where $\gamma_{p,d} = \frac{1}{2}d(p-2)$. 
\end{lemma}
\begin{proof} 
Following the same arguments as in estimate (\ref{7_d7b}),  we can show that  
	\begin{align}
		\big|\int_{F_i}\rho_g\nabla u_g\cdot n_{F_i}\phi_h\,d \sigma \big | \leq 
		C h^{\frac{1+\gamma_{p,d}}{p}}  \|\nabla u_g\|_{L^p(F_i)}  \|\phi_h\|_{dG}.
	\end{align}
	Observing that  $\frac{|r_i|}{h}\sim\frac{d_g}{h}\leq h^{\lambda-1}$, see (\ref{5b_1}),  
	the desired bound follows.  
\end{proof}

\begin{lemma}\label{lemma1}
 Let  Assumption \ref{Assumption_dg_size} hold.  
 Then  there exist a positive constant $C=C(l,p,d)$ such that for all $(u,\phi_h) \in V\times V_h$
  holds  
  \begin{multline}\label{7_e}
  s_{F_l}(u,\phi_h)= \frac{1}{h}\int_{F_l}\phi_h(x_l)  \int_{0}^{1}  \sum_{|\alpha|=l}(x_r-x_l)^\alpha\frac{l}{\alpha !}
         s^{l-1}D^\alpha  u(x_r +s(x_l-x_r))\,d s \,dx_l\leq \\
        C\,d_g^{\,l}\, h^{\zeta}\,d_g^{\,-\frac{p(l-d-1)+1}{p}} \|\phi_h\|_{dG}
       \Big(\int_{\Omega_g}         \kappa_l(z)^p\,dz\Big)^{\frac{1}{p}}, 
  \end{multline}
  where $l\geq 2$,  $\zeta=\frac{-2(p-1)+d(p-2)}{2p}$ and
  $\kappa_l(z) =  \big(\sum_{|\alpha|=l}|D^\alpha u(z)|\big)$.\\
 \end{lemma}
\begin{proof} 
We set  $\frac{1}{q}=\frac{p-1}{p}$ and  $\gamma=\frac{2+d(p-2)}{2p}$. 
We fix an edge  $e_l\subset F_l$ such that   $e_l \subset \partial E_l$, where the 
micro-element $E_l \in T_{h,\Omega_l}^{(l)}$ touches  $F_l$. 
We note that  Assumption \ref{Assumption2} gives $\big|e_l\big|\sim h^{d-1}$.
 Inequality  (\ref{5a_01}) yields  
\begin{multline}\label{7_e_0}
s_{e_l}(u,\phi_h)=
\frac{1}{h^{\frac{1}{p}+\frac{1}{q}+\frac{d(p-2)}{2p}}}\int_{e_l}\phi_h(x_l)  \int_{0}^{1} 
h^{\frac{d(p-2)}{2p}}\sum_{|\alpha|=l}(x_r-x_l)^\alpha\frac{l}{\alpha !}
        s^{l-1}D^\alpha  u(x_r +s(x_l-x_r))\,d s \,dx_l \\
=\int_{e_l}\int_{0}^{1}\frac{1}{h^\gamma}\phi_h(x_l) 
\sum_{|\alpha|=l}h^{\frac{-2(p-1)+d(p-2)}{2p}}(x_r-x_l)^\alpha\frac{l}{\alpha !}
        s^{l-1} D^\alpha  u(x_r +s(x_l-x_r))\,d s \,dx_l \\
\leq C_{} \frac{1}{h^\gamma}\Big(\int_{e_l}\int_0^1 |\phi_h|^q\,d s \,d x_l\Big)^{\frac{1}{q}} 
        h^{\zeta}\,d_g^{\,l} 
       \Big(\int_{e_l}\int_0^1 \big(  \sum_{|\alpha|=l}
       s^{l-1} |D^\alpha  u(x_r +s(x_l-x_r))| \big)^p\,d s \,dx_l\Big)^{\frac{1}{p}}. 
 \end{multline} 
Using the discrete  inequalities (\ref{7_d7a_0_1}), we obtain that
\begin{align}\label{7_g}
 \frac{1}{h^\gamma}\Big(\int_{e_l}\int_0^1 |\phi_h|^q\,d s \,d x_l\Big)^{\frac{1}{q}} \leq 
   C_{p,d} \Big(\frac{1}{ h} \int_{e_l}\phi_h^2 \,d x_l\Big)^{\frac{1}{2}}=
 C_{p,d}{h^{\frac{-1}{2}}}\|\phi_h\|_{L^2(e_l)}.
\end{align}
It remains to estimate the second term in (\ref{7_e_0})
on every $e_l$.  
 By the change of variables $z=x_r+s(x_l-x_r)$, we have   that
 \begin{subequations}\label{7_f}
 \begin{align}
 (x_r-z)s^{-1} &= (x_r-x_l),\\
 \label{7_f_1}
 \big(x_{l_1},x_{l_2},...,x_{l_d}\big)&=\big(x_{r_1}(x_{l}),x_{r_2}(x_l),...,x_{r_d}(x_l)\big) +\\
 \nonumber
 &                                      s^{-1}\big(z_1,z_2,...,z_d\big)-\big(x_{r_1}(x_{l}),x_{r_2}(x_l),...,x_{r_d}(x_l)\big) \\
 (z-x_r)^\alpha s^{-l}&=(x_r-x_l)^\alpha,\\
 \det\Big(\frac{\partial(x_l,s)}{\partial (z,s)}\Big)&=s^{-d}.
 \end{align}
\end{subequations}
\begin{remark}\label{remark_4}
	The previous relations have been given for a general gap region. 
	For the gaps  described in Subsection \ref{gab_region}, see Fig. \ref{Sub_doms_gap},
	the relations (\ref{7_f}) take the form  
	 $$
	 \begin{array}{ll}
          x_{r_1}(x_l):=& x_{r_1}=x_{l_1}  \\
          x_{r_2}(x_l):=& x_{r_2}=x_{l_2} \\
          x_{r_3}(x_l):=& x_{r_3}(x_{l_1},x_{l_2})=d_g\,\zeta(x_{l}),
\end{array} 
$$
     	and it can be verified that  $ \det\Big(\frac{\partial(x_l,s)}{\partial (z,s)}\Big)=s^{-1}$.
\end{remark}

Now, let $z$ be any point in the interval $(x_l, x_r)$. Moreover, let 
$x_{l_0}$ and   $x_{r_0}=\mathbf{\Phi}_{l,r}(x_{l_0})$ be as in Proposition \ref{Prop2_0}. Then 
we can deduce that  
 \begin{align}\label{7_f_0}
 |z-x_r|\leq  |z-x_{r_0}|+|x_r-x_{r_0}|,\quad  |x_r-x_l| = C d_g.
 \end{align}
Now, since the parameter  $s$ varies between 0 and 1, 
  the variable $z$ also runs in the region $E_g \subset \Omega_g$ with 
  $\big|E_g|\leq d_g h^{d-1}$, see Fig.  \ref{Sub_doms_gap}(c).
  Describing the domain with respect to $(z,s)$ variables, 
  the range of $s$ is defined to be such that the  $\omega(z,s):=(z-x_r)\frac{1}{s}+x_r$,
  should remain in $e_l$. 
  Using (\ref{7_f}),  (\ref{7_f_0}) and   the fact that 
  that $diameter(e_l)\sim h$ the new variables satisfy 
  \begin{align}\label{7_gba1}
  	\big|(z-x_r)\frac{1}{s}+x_r\big |  \leq C_{\theta} h,\,\, \text{and}\,\,
s\leq\frac{|z-x_{r_0}|+|x_r-x_{r_0}|}{C d_g}=\frac{C_{\Omega_g}}{d_g},
  \end{align}
  where the  constant $C_{\theta}>0$ depends on the quasi-uniformity properties of the meshes, see Assumption \ref{Assumption2}, and 
  $C_{\Omega_g}>0$ on the shape of $\Omega_g$. 
Thus,  by the  change of the order of integration and by the change of variable  on the second 
term in (\ref{7_e_0}),  we get 
    \begin{multline}\label{7_gb}
   \Big(\int_{e_l}\int_{0}^{1} \big( \sum_{|\alpha|=l}s^{l-1}
            |D^\alpha  u(x_r +s(x_l-x_r))|  \big)^p\,d s \,dx_l\Big)^{\frac{1}{p}} = 
            \Big(\int_{0}^1\int_{\omega(z,s)=x_l} 
            \big(s^{l-d-1}\sum_{|\alpha|=l} |D^{\alpha}u(z)| \big)^p\,dz\,d s\Big)^{\frac{1}{p}}  \\          
 \leq C\Big(\int_{E_g}\int_{0}^{\frac{C_{\Omega_g}}{d_g}}
  \big(s^{l-d-1}\sum_{|\alpha|=l}|D^{\alpha}u(z)| \big)^p  \,d s\,dz\Big)^{\frac{1}{p}} \leq       
  C \Big( \int_{E_g}    s^{p(l-d-1)+1}\Big|_{0}^{\frac{C_{\Omega_g}}{d_g}} 
        \big(  \sum_{|\alpha|=l}|D^\alpha  u(z)| \big)^p\,dz \Big)^{\frac{1}{p}}\\
        \leq       C   \Big( \int_{E_g}   \Big(\frac{C_{\Omega_g}}{d_g}\Big)^{p(l-d-1)+1}
        \kappa_l(z)^p \,dz\Big)^{\frac{1}{p}} \leq 
     C d_g^{-\frac{(p(l-d-1)+1)}{p}}    \Big(\int_{E_g}        \kappa_l(z)^p \,dz\Big)^{\frac{1}{p}}.
      \end{multline}
Finally, inserting (\ref{7_g}) and (\ref{7_gb}) into (\ref{7_e_0}), and then summing over 
        all $e_l \subset F_l$, we obtain 
  \begin{align}\label{7_ga2}
 s_{F_l}(u,\phi_h)\leq  C\Big(\sum_{e_l \subset F_l}\big(\frac{1}{ h^{\frac{1}{2}}} \|\phi_h\|_{L^2(e_l)}\big)^q\Big)^{\frac{1}{q}}
d_g^{\,l} h^{\zeta} d_g^{-\frac{(p(l-d-1)+1)}{p}} 
  \Big(\sum_{E_g \subset \Omega_g}\int_{E_g}         \kappa_l(z)^p \,dz\Big)^{\frac{1}{p}}.
  \end{align} 
  Using the fact that  the 
$
f(x) = (\eta_0\alpha^x+\eta_0\beta^x)^{\frac{1}{x}}, {\ }\eta_0>0,x>2
$ is decreasing, we have the inequality
\begin{align}\label{7_ga3}
	\Big(\sum_{e_l \subset F_l}\big(\frac{1}{ h^{\frac{1}{2}}} \|\phi_h^2\|_{L^2(e_l)}\big)^q\Big)^{\frac{1}{q}} \leq 
	C\Big(\frac{1}{h}\|\phi_h\|_{L^2(F_l)}^2\Big)^{\frac{1}{2}}\leq C \|\phi_h\|_{dG}.
\end{align}

We insert (\ref{7_ga3}) into (\ref{7_ga2}), and then  we  deduce (\ref{7_e}).
\end{proof}
\par

Working in a similar way as in the proof of Lemma \ref{lemma1}, we can show similar
bounds for the other remainder  terms, i.e.,
\begin{multline}\label{7_i}
s_{F_r}(u,\phi_h)=
\frac{1}{h}\int_{F_r}\phi_h(x_r)  \int_{0}^{1}  
\sum_{|\alpha|=l}(x_l-x_r)^\alpha\frac{l}{\alpha !}
        sD^\alpha  u(x_l +s(x_r-x_l))\,d s \,dx_r \\
\leq        C\|\phi_h\|_{dG} 
      d_g^{\,l} h^{\zeta} d_g^{-\frac{(p(l-d-1)+1)}{p}}
      \Big(\int_{\Omega_g}         \kappa_l(z)^p\,dz\Big)^{\frac{1}{p}}.
 \end{multline}
We continue to give an estimate for the  $R_{\Omega_g}(\cdot,\cdot)$ defined in 
(\ref{7_d3_b}).
\begin{lemma}\label{lemma4}
Under the Assumptions  \ref{Assumption2} and \ref{Assumption_dg_size},  there exist a positive constant 
$C=C(\rho, p,d,l)$, such 
that the estimate 
 \begin{flalign}\label{7_i0}
 |R_{\Omega_g}(u,\phi_h)| \leq C 
  \|\phi_h\|_{dG} \big(\|\nabla u_g\|_{L^p(\partial \Omega_g)}+ \|\kappa_2\|_{L^p(\Omega_g)}\big)
   h^{\beta}, 
   \end{flalign}
   holds true for all $(u,\phi_h) \in V\times V_h$,
  where $\kappa_2 =  \big(\sum_{|\alpha|=2}|D^\alpha u|\big)$,  
   $\zeta=\frac{-2(p-1)+d(p-2)}{2p}$, and \\
  $\beta=\min\{2\lambda+\zeta-\frac{p(1-d)+1}{p},\lambda-1+\frac{1+\gamma_{p,d}}{p},
  1+\zeta+\lambda-\frac{p(1-d)+1}{p}\}$,
 \end{lemma}
\begin{proof}
Clearly,  Proposition \ref{Prop2_0} in combination with Assumption \ref{Assumption2} 
imply that   $|r_l|\sim h^{\lambda}$ and $|r_r| \sim h^{\lambda}$.
Recalling the definition of  $R_{\Omega_g}(\cdot,\cdot)$, see (\ref{7_d3_b}), and using the estimates
(\ref{7_00_e}), (\ref{7_e}) and (\ref{7_i}) with $|\alpha|=2$,  we can derive 
\begin{multline}\label{7_j}
|R_{\Omega_g}(u,\phi_h)| \leq C\Big( h^{\lambda-1}h^{\frac{1+\gamma_{p,d}}{p}}
  \|\nabla u_g\|_{L^p(\partial \Omega_g)}  \\
+    d_g^{\,2}\, h^{\zeta}\,d_g^{\,-\frac{p(2-d-1)+1}{p}} 
              \|\kappa_2\|_{L^p(\Omega_g)}  \\
+    \,d_g^{\,2}\, h^{\zeta+1-\lambda}\,d_g^{\,-\frac{p(2-d-1)+1}{p}}
      \|\kappa_2\|_{L^p(\Omega_g)}\Big)  \|\phi_h\|_{dG},
\end{multline}
where the constant  depends on the constant in (\ref{7_00_e}),  the constant in (\ref{7_e}) and
the quasi-uniformity parameters of the mesh, see Assumption \ref{Assumption2}.
Setting   $d_g\sim h^\lambda$ in (\ref{7_j}), we immediately arrive at estimate  (\ref{7_i0}).
\end{proof}
\section{Error estimates}
Next, we give an error estimate by means of a variation of Cea's Lemma applied in dG frame. 
We  use the estimate for $|R_{\Omega_g}(u,\phi_h)|$ as is given in (\ref{7_i0}). 
The linearity of the $B_h(\cdot,\cdot)$, see (\ref{7_d5}) and  (\ref{7_d3_b}), and the discrete variational
form (\ref{7_d4})  yield
\begin{equation}\label{4.5_a}
	B_h(u_h-z_h,\phi_h) = F_h(\phi_h)-B_h(z_h,\phi_h), \quad \text{for all} \quad \phi_h, {\ }z_h\in V_h.
\end{equation}
Using (\ref{7_d3_b}), (\ref{7_d5}), (\ref{7_d6}) and (\ref{7_d3_a}), we get 
\begin{multline}\label{4.5_b}
	B_h(u_h-z_h,\phi_h) = B(u,\phi_h) +
	\sum_{i=l,r}\frac{\rho_i}{h}\int_{\partial\Omega_i \cap \partial \Omega}(u-u_D) \phi_h\,d\sigma 
		-B_h(z_h,\phi_h) + F_h(\phi_h)-l_{f,\setminus\overline{\Omega}_g}(\phi_h)  \\
=	B_h(u,\phi_h)+	R_{\Omega_g}(u,\phi_h) -
	\sum_{i=l,r}\frac{\rho_i}{h}\int_{\partial\Omega_i \cap \partial \Omega}u_D \phi_h\,d\sigma 
	-B_h(z_h,\phi_h)+ \sum_{i=l,r}\frac{\rho_i}{h}\int_{\partial\Omega_i \cap \partial \Omega}u_D \phi_h\,d\sigma\\
	=B_h(u-z_h,\phi_h) +  R_{\Omega_g}(u,\phi_h).
\end{multline}
We choose in (\ref{4.5_b}) $\phi_h=u_h-z_h$. Then, Lemma \ref{lemma_0} and Lemma \ref{lemma00} imply
\begin{multline}\label{4.5_c}
	C_m\|u_h-z_h\|^2_{dG}\leq C_M \big(\|u-z_h\|_{dG}^p 
+   h^{1+\gamma_{p,d}}\| \nabla (u-z_h)\|^p_{L^p(\partial \Omega_g)}\big)^{\frac{1}{p}}
   \|u_h-z_h\|_{dG}+ |R_{\Omega_g}(u,u_h-z_h)| \\
       \leq C_M \big(\|u-z_h\|_{dG}^p + 
   h^{1+\gamma_{p,d}}\| \nabla (u-z_h)\|^p_{L^p(\partial \Omega_g)}\big)^{\frac{1}{p}}\|u_h-z_h\|_{dG} 
+ C_1        \|u_h-z_h\|_{dG}    h^{\beta} \mathcal{K}_{p},
\end{multline}
where we previously used the estimate (\ref{7_i0}) 
and 
$\mathcal{K}_{p}=\|\nabla u_g\|_{L^p(\partial \Omega_g)} + \|\kappa_2\|_{L^p(\Omega_g)}$.
Applying triangle inequality in (\ref{4.5_c}), 
we can easily arrive at the following estimate 
\begin{align}\label{4.5.d}
	\|u-u_h\|_{dG} \leq C \Big( \big(\|u-z_h\|_{dG}^p + 
   h^{1+\gamma_{p,d}}\| \nabla (u-z_h)\|^p_{L^p(\partial \Omega_g)}\big)^{\frac{1}{p}}
+	h^{\beta}\,\mathcal{K}_{p}\Big),
\end{align}
where the constant $C$ is specified by the constants appearing in (\ref{4.5_c}). 
\par
Now,  we can prove the main error estimate result of the section. 
Such an estimate requires  quasi-interpolation estimates  of B-splines. 
 By  the results of multidimensional B-spline interpolation, 
 (see \cite{LT:Shumaker_Bspline_book} and \cite{LT:LangerToulopoulos:2014a}), 
 we can construct a quasi-interpolant
 $\Pi: W^{l, p }\rightarrow V_h$ with $l\geq 1, p>1 $,  
 such that the following interpolation estimates to be true.

\begin{lemma}\label{lemma5.3}
 Let $u\in W^{l,p}(\Omega_i)$ with  
 $l\geq 2,$ $p\in (\max\{1,\frac{2d}{d+2(l-1)}\},2]$.  Then  for $i=l,r$,  
  there exist  constants
  $C_i$  such that
 \begin{subequations}\label{5.12a}
 \begin{align}
 \label{5.12a_a}
 h^{\frac{1+\gamma_{p,d}}{p}}\| \nabla (u-\Pi u)\|_{L^p(\partial\Omega_i)}
 	\leq &   C_i h^{\delta_{\Pi} (l,p,d)}\|u\|_{W^{l,p}(\Omega_i)},\\
 \label{5.12a_b}
 \frac{\{\rho_i\}}{h}\big(\|u-\Pi u\|_{L^2(\partial \Omega_i)}\big)^2 \leq & C_i 
        \big( h^{\delta_{\Pi} (l,p,d)} \|u\|_{W^{l,p}(\Omega_i)}\big)^2,
	\end{align}
\end{subequations}
	 where $\delta_{\Pi} (l,p,d)= l+(\frac{d}{2}-\frac{d}{p}-1)$ and $\gamma_{p,d} = \frac{1}{2}d(p-2)$. 
 \begin{proof}
 	The proofs are given in \cite{LT:LangerToulopoulos:2014a} for the  general case of non-matching grids.
 	 \end{proof}

\end{lemma}

\begin{lemma}\label{lemma7}
Let $u$ satisfy Assumption \ref{Assumption1}.  Then,
there exist  constants $C_i>0$ with $i=l,r$ independent of the grid sizes $h$ 
such that 
 \begin{align}\label{4.5.d_1}
 	\|u-\Pi u\|_{dG} & \leq   \sum_{i=l,r} C_i h^{\delta_{\Pi} (l,p,d)}\|u\|_{W^{l,p}(\Omega_i)},
 \end{align}
  where $\delta_{\Pi} (l,p,d)= l+(\frac{d}{2}-\frac{d}{p}-1)$.
\end{lemma}
 \begin{proof}
 	We show first an estimate for  $E\in T_{h,\Omega_i}^{(i)}$ for $i=l,r$. We associate with each
 	$E\in T_{h,\Omega_i}^{(i)}$ the local support extension $D^{(i)}_{{E}}$ of the  B-splines, see Subsection \ref{Bsplinespace}.
 	By the properties of the quasi-interpolant $\Pi$ we have the estimate, see  \cite{LT:LangerToulopoulos:2014a},  
\begin{align}\label{4.5.d_2}  
                             |u-\Pi u|_{W^{1,p}({E})} \leq  C h^{l-1}\|u\|_{W^{l,p}(D^{(i)}_{{E}})}.
\end{align}
We quote below an inequality which holds for $f$ satisfying Assumption \ref{Assumption1} and   has been shown in \cite{LT:LangerToulopoulos:2014a},
\begin{align}\label{4.5.d_3}
\|f\|_{L^2(E)} \leq C_{i} h^{\frac{d}{2}-\frac{d}{p}}\big( \|f\|^p_{{L^p(E)}} + h^p|f|^p_{W^{1,p}(E)}\big)^\frac{1}{p},{\ }
\text{for }{\ }E\in T_{h,\Omega_i}^{(i)},{\ } i=l,r.
\end{align}
Setting $f:=\nabla u-\nabla \Pi u$ in (\ref{4.5.d_3}), summing over all micro-elements and
applying the approximation estimate (\ref{4.5.d_2}), we obtain that

\begin{align}\label{4.5.d_4}
	 |u-\Pi u|^2_{W^{1,2}(\Omega_i)} \leq 
	 C_i \big(h^{l+(\frac{d}{2}-\frac{d}{p}-1)}\|u\|_{W^{l,p}(\Omega_i)}\big)^2,\, \text{for}\, i=l,r.	 
\end{align}
It remains to estimate the jump terms in dG-norm. We apply  (\ref{5.12a_b}) and get
	\begin{align}\label{4.5.d_01}
	\sum_{i=l,r} \frac{\rho_i}{h}\|u_i-\Pi u_i\|^2_{L^2(\partial \Omega_i\cap \partial\Omega)} +
	\frac{ \{\rho\} }{h} \| u_i-\Pi u_i\|^2_{L^2(F_i)}
\leq \sum_{i=l,r}C_i         \big( h^{\delta_{\Pi} (l,p,d)} \|u\|_{W^{l,p}(\Omega_i)}\big)^2,
	\end{align}

Recalling the definition of $\|.\|_{dG}$, combining the estimates (\ref{4.5.d_4}) and (\ref{4.5.d_01})
we can derive (\ref{4.5.d_1}). 
 \end{proof}

\par
\begin{theorem}\label{Theorem_1_estimates}
Let $u$ be the solution of problem (\ref{7_d3_a}),  $u_h$ be the corresponding
dG IgA solution of problem (\ref{7_d4}), and let  $d_g = h^\lambda$ with $\lambda \geq 1$.  Then the  error estimate
	\begin{align}\label{4.5_e}
		\|u-u_h\|_{dG} \lesssim  h^{\delta_{\Pi} (l,p,d)} \sum_{i=l,r} 	\|u\|_{W^{l,p}(\Omega_i)} +
	 h^{\beta}\,\mathcal{K}_{p},
	\end{align}
holds,  where $\delta_{\Pi} (l,p,d)= l+(\frac{d}{2}-\frac{d}{p}-1)$,
 $\mathcal{K}_{p}=\|\nabla u_g\|_{L^p(\partial \Omega_g)} + \|\kappa_2\|_{L^p(\Omega)}$, 
 $\kappa_2 =  \big(\sum_{|\alpha|=2}|D^\alpha u|\big)$, \\ 
  $\beta=\min\{2\lambda+\zeta-\frac{p(1-d)+1}{p},\lambda-1+\frac{1+\gamma_{p,d}}{p},
  1+\zeta+\lambda-\frac{p(1-d)+1}{p}\}$, 
   $\zeta=\frac{-2(p-1)+d(p-2)}{2p}$, $\gamma_{p,d} = \frac{1}{2}d(p-2)$ and the positive 
   constants $C_i$ are the same as in (\ref{4.5.d_1})
\end{theorem}
\begin{proof}
	The required estimate  follows easily by   introducing
	the quasi-interpolation estimates  (\ref{5.12a_a}) and (\ref{4.5.d_1}) into estimate  (\ref{4.5.d}). 
	\end{proof}

\newpage
 \section{Numerical tests}
 \label{Section_numerics}
 \begin{wraptable}[28]{r}[22pt]{4.cm}
 %
 \begin{tabular}{|c||c|}
  \hline
             \multicolumn{1}{|c||}{$\mathbf{\Phi}_l$}& \multicolumn{1}{|c|}{$\mathbf{\Phi}_r$} \\  \hline  
 (-1, -0.2)     & (0,    0) \\
 (-0.75, 0)     & (0.25, 0)\\
 (-0.5, 0)      & (0.5,  0)\\
 (-0.25, 0)     & (0.75, 0)\\  
 (0, 0)         & (1,    0.2) \\ \hline
 (-1, 0.25)     & (0,    0.25) \\
 (-0.75, 0.25)  & (0.25, 0.25) \\
 (-0.5, 0.25)   & (0.5,  0.25) \\
 (-0.25, 0.25)  & (0.75, 0.25) \\
 (0, 0.25)      & (1,    0.25) \\ \hline 
 (-1, 0.5)      & (0,    0.5) \\ 
 (-0.75, 0.5)   & (0.25, 0.5) \\
 (-0.5, 0.5)    & (0.5, 0.5) \\
 (-0.25, 0.5)   & (0.75, 0.5) \\
 (0, 0.5)       & (1,    0.5) \\  \hline
 (-1, 0.75)     & (0,    0.75) \\
 (-0.75, 0.75)  & (0.25, 0.75) \\
 (-0.5, 0.75)   & (0.5,   0.75) \\
 (-0.25, 0.75)  & (0.75, 0.75) \\
 (0, 0.75)      & (1,    0.75) \\  \hline
 (-1, 1.2)      & (0,1) \\
 (-0.75, 1)     & (0.25,1) \\
 (-0.5, 1)      & (0.5, 1) \\
 (-0.25, 1)     & (0.75, 1)\\
 (0, 1)         & (1, 0.8) \\ \hline
 \end{tabular}
 \caption{The control points for the  mappings $\mathbf{\Phi}_i,\,i=l,r$.}
 \label{contrl_point_omega_lr}
\end{wraptable} 
We have performed several numerical tests in order to confirm the theoretically predicted
order of accuracy for the dG IgA scheme 
proposed in (\ref{7_d5}).
We  will discuss two- and three- dimensional 
test examples. 
All tests have been performed in G+SMO\footnote{G+SMO: https://www.gs.jku.at/trac/gismo}, 
which is a generic object-oriented C++ library for IgA computations,
see also \cite{HLT:JuettlerLangerMantzaflarisMooreZulehner:2014a,HLT:LangerMantzaflarisMooreToulopoulos:2015a}.
We have used second order ($k=2$) B-spline spaces for all tests.
Every example has been solved
 applying several mesh refinement steps with $h_i,h_{i+1},...,$ satisfying Assumption \ref{Assumption2}.
 The numerical convergence rates $r$ have been  computed by the 
 ratio 
$r = {\ln(e_i/e_{i+1})}/{\ln (h_i /h_{i+1} )}, \,i=1,2,...$, 
where the error $e_i:=\|u-u_h\|_{dG}$ is always computed  
on the  meshes $T^{(l)}_{h_i,\Omega_l}\cup T^{(r)}_{h_i,\Omega_r}$.
We mention that, in the test cases with highly smooth solutions, i.e., $k+1 \leq l$, 
   the approximation order in (\ref{4.5_e}) becomes $\delta_\Pi(l,p,d) = k$.
\par
The code that has been materialized for performing the tests uses uni-directional Taylor expansions, see
(\ref{parmatric_rl}), Remark \ref{remark_1} and Remark \ref{remark_4}. 
The predicted values of  power  $\beta$ in (\ref{4.5_e})  are given in Table \ref{table_value_r}.
\par
For the two dimensional examples, we use the knot vectors $\Xi_{i}^1=\Xi_{i}^2:=
\{0 , 0 ,0 ,0.5, 0.5 ,1, 1, 1\}$, with $i=l,r$, to define
the parametric mesh and to construct the corresponding second order B-spline space, see (\ref{0.00b1}). 
The B-spline parametrizations of $\Omega_l$ and $\Omega_r$, see (\ref{0.0c}), are constructed using 
the control points which are listed
in Table \ref{contrl_point_omega_lr}.
In any test case, the  gap region  is artificially created by moving 
all 
control points of the second subdomain which have   the form $(0,\xi)$, where $0<\xi<1$,
in the direction $(1,0)$,
For all tests, the parametric mapping in (\ref{parmatric_lr_1}) has the form 
$\mathbf{\Phi}_{l,r}(x_1,x_2)= (x_1,x_2)+d_g 4x_2(1-x_2)(1,0)$.
\begin{table}
\begin{minipage}[b]{.45\textwidth}
  \centering
  \begin{tabular}{|c||c|c|c|}
 \hline
  & \multicolumn{3}{|c|}{B-spline degree $k=2$   }\\ \hline  
  & \multicolumn{3}{|c|}{Smooth solutions, $u\in W^{l>3,p=2}$}  \\ \hline
   $d_g=h^\lambda$       &$\lambda=1$& $\lambda=2$ &$\lambda=3$\\\hline
  $\beta:=$              &0.5        &  1.5        & 2.5      \\   \hline
  $\delta_\Pi(l,p,d):=$              &2        &  2        & 2      \\  
 \hline
\end{tabular}
\caption{The values of the order $\beta$ of the remainder terms estimates and the B-spline
         approximation order.}
\label{table_value_r}
\end{minipage}\qquad\quad
\begin{minipage}[b]{.45\textwidth}
  \begin{tabular}{ |c||c|c|c|  }
 \hline
          - & \multicolumn{3}{|c|}{$ d_g=h^\lambda$} \\  \hline
      -               & $\lambda=1$ & $\lambda=2$ &$\lambda=3$ \\  \hline
          - & \multicolumn{3}{|c|}{expected rates $r$} \\  \hline
    $\gamma=0.42,\, u\in W^{2,1.26}$  &0.22         &  0.42          & 0.42            \\  \hline
    $\gamma=0.38,\, u\in W^{2,1.23}$  &0.19         &  0.38          & 0.38            \\  \hline
    $\gamma=1,\, u\in W^{2,2}$       &0.5          &  1          & 1            \\  \hline
    $\gamma=1.5,\, u\in W^{2.5,2}$   &0.5          &  1.5        & 1.5            \\  \hline
    $\gamma=2,\, u\in W^{3,2}$       &0.5          &  1.5        & 2          \\  \hline
\end{tabular}
\caption{The expected values of the 
             rates $r$   for the example
          with           low regularity solution.}
\label{rate_gamma_lambda}
\end{minipage}
\end{table}
 \subsection{Two-dimensional numerical examples}

\paragraph{Example  1} The domain  $\Omega$ with the subdomains
$\Omega_l,\,\Omega_r$ and $\Omega_g$ are shown in Fig.~\ref{Fig1_Test_1Gaps}(a). The Dirichlet boundary
condition and the right hand side $f$ are determined by the exact solution 
$u(x_1,x_2)=\sin(5\pi x_1)\sin(4\pi x_2)$. 
 In this example, we consider 
the homogeneous diffusion case, i.e.,
$\rho_l=\rho_r=1$, and the left interface is given by
 	$F_l=\{(x_{l,1},x_{l,2}):x_{l,1}=0, 0\leq x_{l,2} \leq 1\}$, 
 	see Fig.  \ref{Fig1_Test_1Gaps}(a). We performed two  computations.
        In the first computation, the size of $d_g$ was successively defined to be $\mathcal{O}(h^\lambda)$, 
        with $\lambda=1,2$ and $3$.   The numerical convergence rates for several levels of mesh refinement are 
 	plotted in Fig. \ref{Fig1_Test_1Gaps}(b). 
	They are in good
	agreement with our theoretically predicted estimates 
 	given in Theorem \ref{Theorem_1_estimates}, see also Table \ref{table_value_r}.  
 	 In the second computation, we progressively decrease    the size of $d_g$ 
	 when we are performing the computation on  successively  refined meshes. 
 	 In Fig.  \ref{Fig1_Test_1Gaps}(c), we present the corresponding convergence rates $r$. 
 	 For the first  meshes, we set $d_g = h$, and the rates  are $r=0.5$, for the next,  we set
 	 $d_g=h^2$ and the rates are increased to $r=1.5$, behaving according to the  rates 
	predicted by the theory.
 	Finally, for the last refinement levels, we set $d_g=h^3$, and  the rates become optimal 
 	having similar behavior as in Fig. \ref{Fig1_Test_1Gaps}(b). 
 
  \begin{figure}[h]
  \begin{subfigmatrix}{3}
\subfigure[]{\includegraphics[width=4.0cm, height=4.75cm]{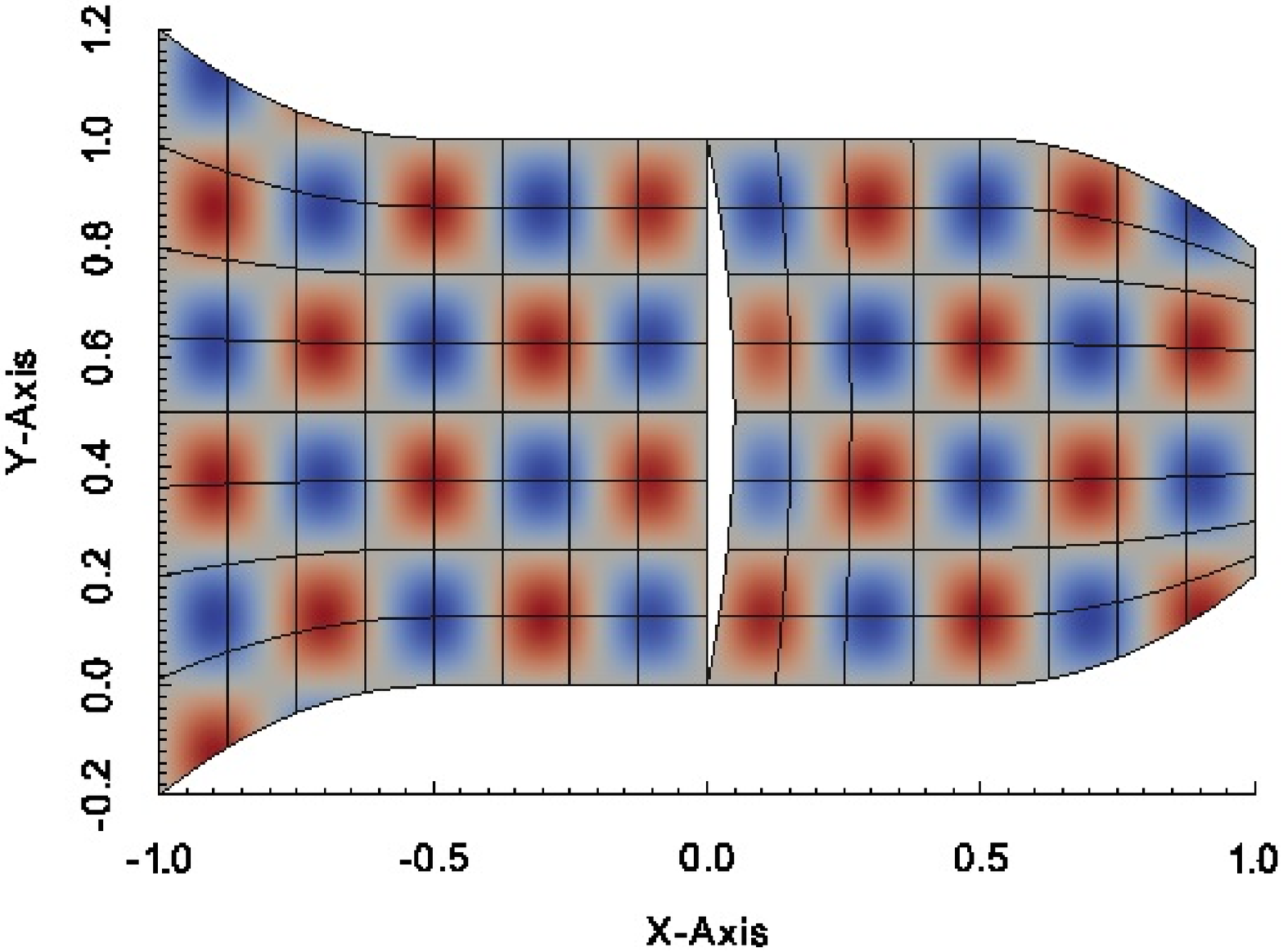}}
\subfigure[]{\includegraphics[width=4cm, height=4cm]{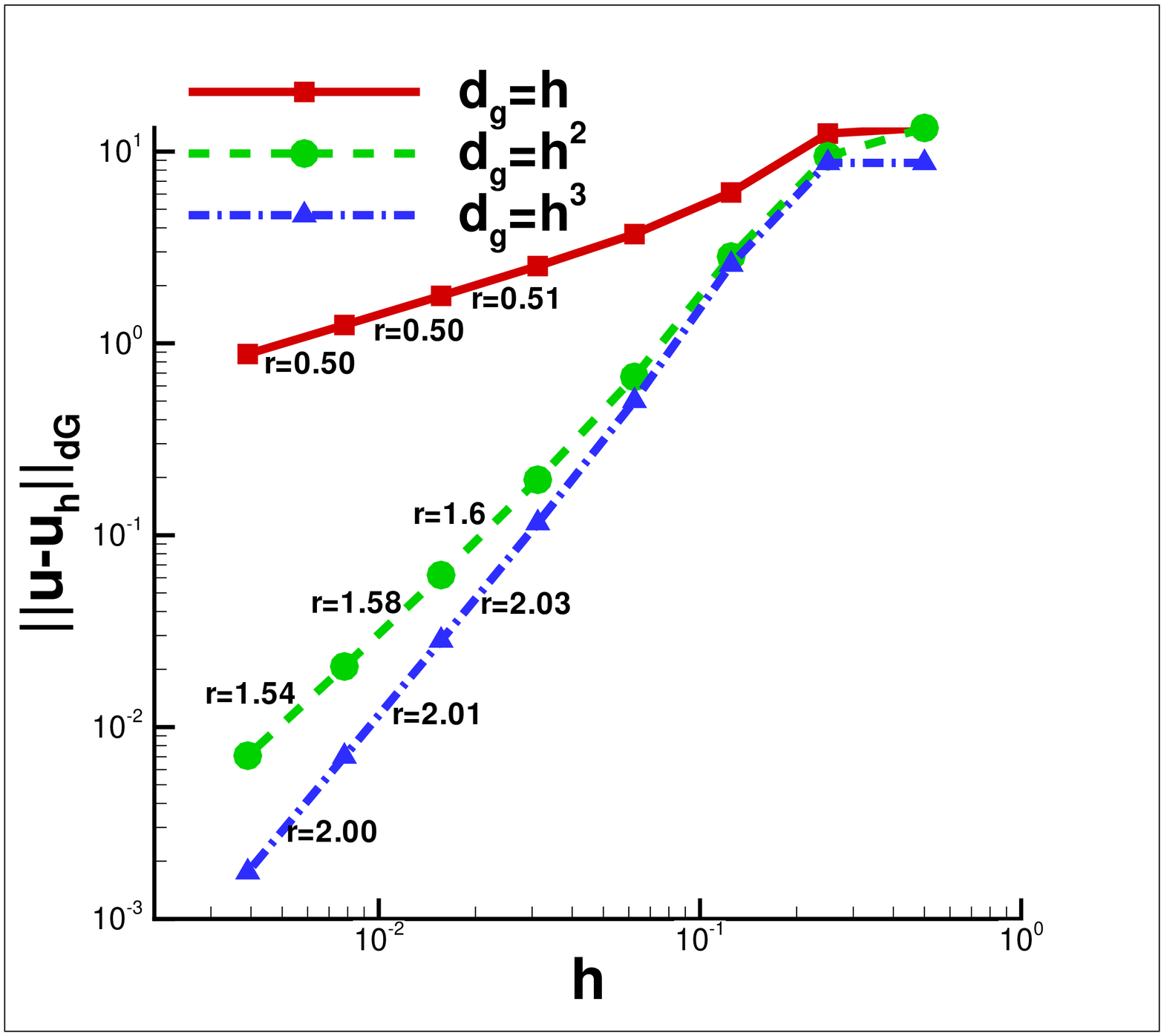}}
\subfigure[]{\includegraphics[width=4cm, height=4cm]{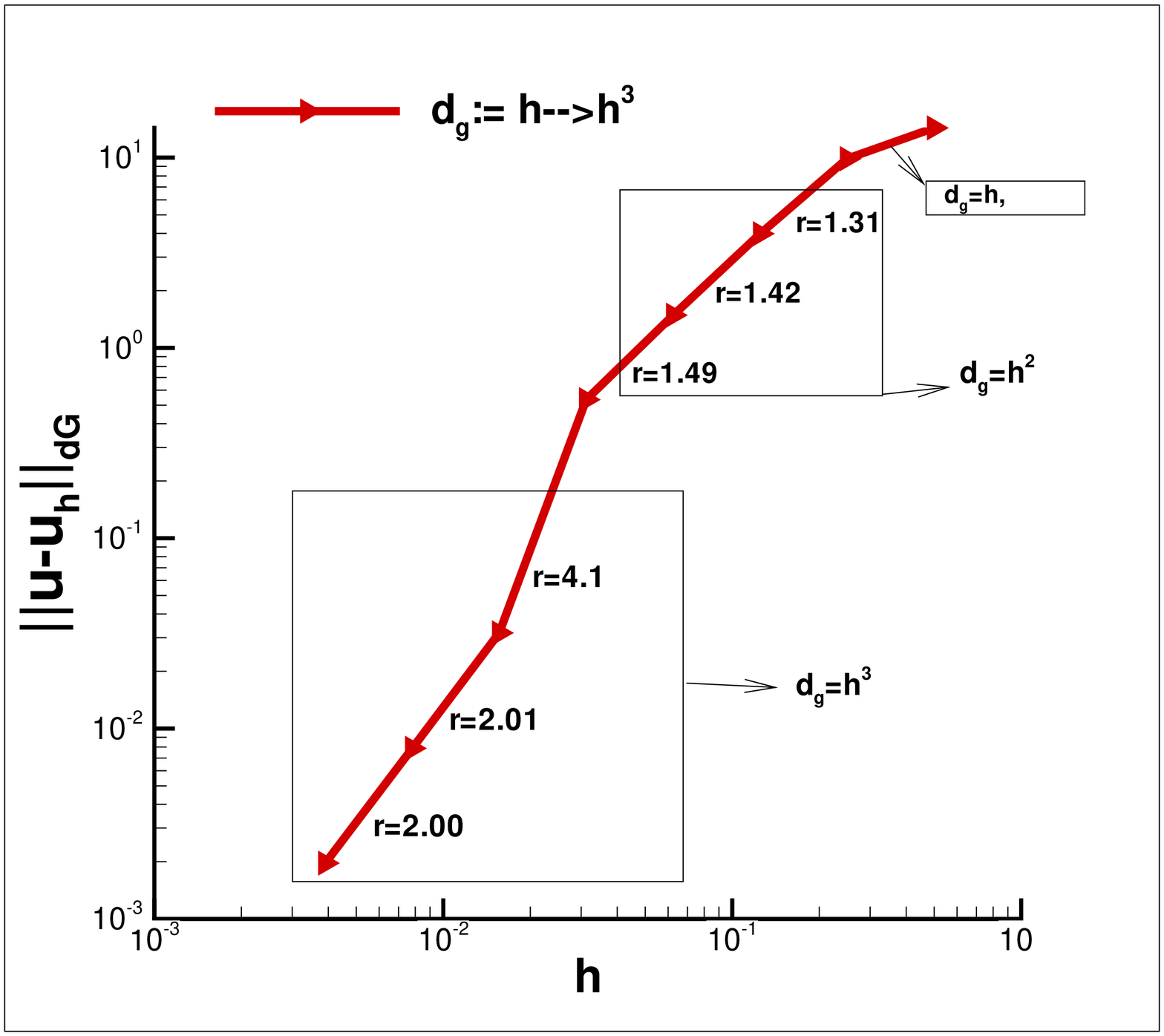}}
 \end{subfigmatrix}
  \caption{Example 1: (a) The subdomains $\Omega_l,\,\Omega_r$ and $\Omega_g$ and the exact solution contours, 
                      (b) The convergence rates for  $d_g=\mathcal{O}(h^\lambda)$,
                      (c) The convergence rates for  fixed $d_g$. }
  \label{Fig1_Test_1Gaps}
\end{figure}
%
\paragraph{Example  2} We consider the problem with discontinuous coefficient, i.e.,  we set 
 	              $\rho_1=4\pi$ in $\Omega_1$ and $\rho_2=1$ in $\Omega_2$.  The 
 	                  domain $\Omega$ is presented  in Fig. \ref{Fig1_Test_1Gaps}(a), where the interface $F$ is the $x_2$-axis.
 	                  The solution is given by the formula
 	                  \begin{align}\label{NE_1}
 	                  u(x_1,x_2)=
 	                  \begin{cases}
                             \exp{(x_1)}-1 & \text{if } (x_1,x_2)\in \Omega_1, \\
                             \sin(4\pi x_1) & \text{if} {\ }(x_1,x_2)\in \Omega_2.
                          \end{cases}
                          \end{align}
                          The  boundary conditions and the source function $f$ are 
 	                  determined by (\ref{NE_1}). 
                  Note that in this test case, we have  $\llbracket u \rrbracket |_{F}=0$ as well $\llbracket \rho \nabla u \rrbracket |_{F}=0$
                  for the normal flux on $F$.            
                  The contours of the exact solution on the  domain $\Omega$ are presented in Fig. \ref{Fig2_Test_2Gaps}(a). 
                  The problem has been solved on  meshes refined following a sequential process, where
                  we set  $d_g=h^\lambda$, with                   $\lambda=1,2$ and $3$. 
                  Thus for every computation the gap boundary is formed by the choice of $h$ and $\lambda$. 
                  In Fig. \ref{Fig2_Test_2Gaps}(b), we plot the $u_h$ solution on $\Omega\setminus \Omega_g$
                  computed on a grid with $h=0.125$ and $d_g=0.1$.  
                  The computed rates are presented in                   Fig. \ref{Fig2_Test_2Gaps}(c). 
                    For the cases where $\lambda=1$ and $\lambda=2$, we observe that 
                    the values of the rates behave according to the predicted rates, see (\ref{4.5_e}).
                   The error corresponding to the $d_g=h^3$ test case (dashed dot line) on the first
                   refinements appears  to decay slower than it was expected, but finally
                   on the last refinement levels tends to take the optimal value, which has predicted by the theory.
                  By this example we validate numerically  the 
                  predicted convergence rates for problems with
                  discontinuous coefficient and smooth solutions.

 \begin{figure}
  \begin{subfigmatrix}{3}
\subfigure[]{\includegraphics[width=4.0cm, height=4.75cm]{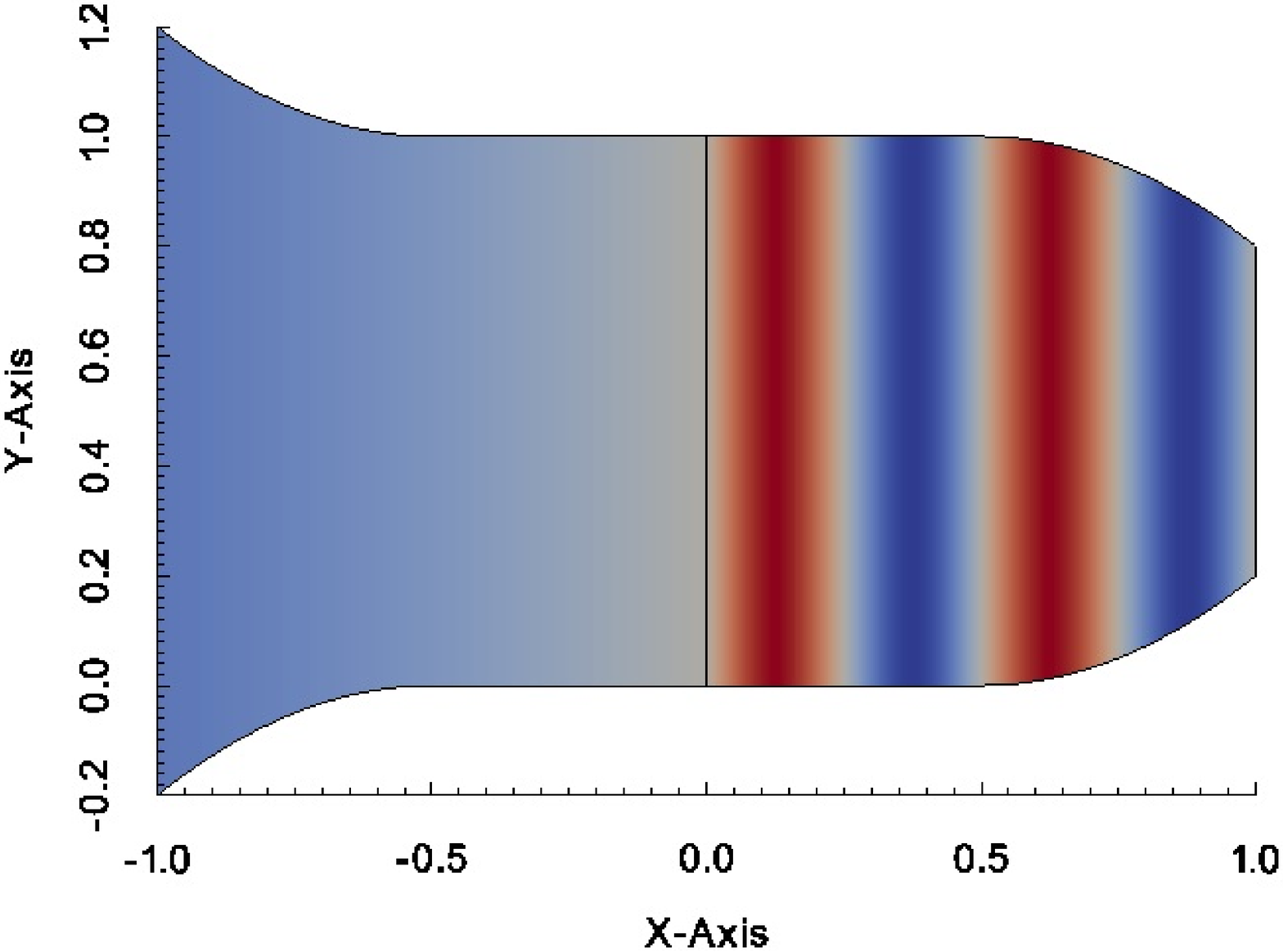}}
\subfigure[]{\includegraphics[width=4cm, height=4cm]{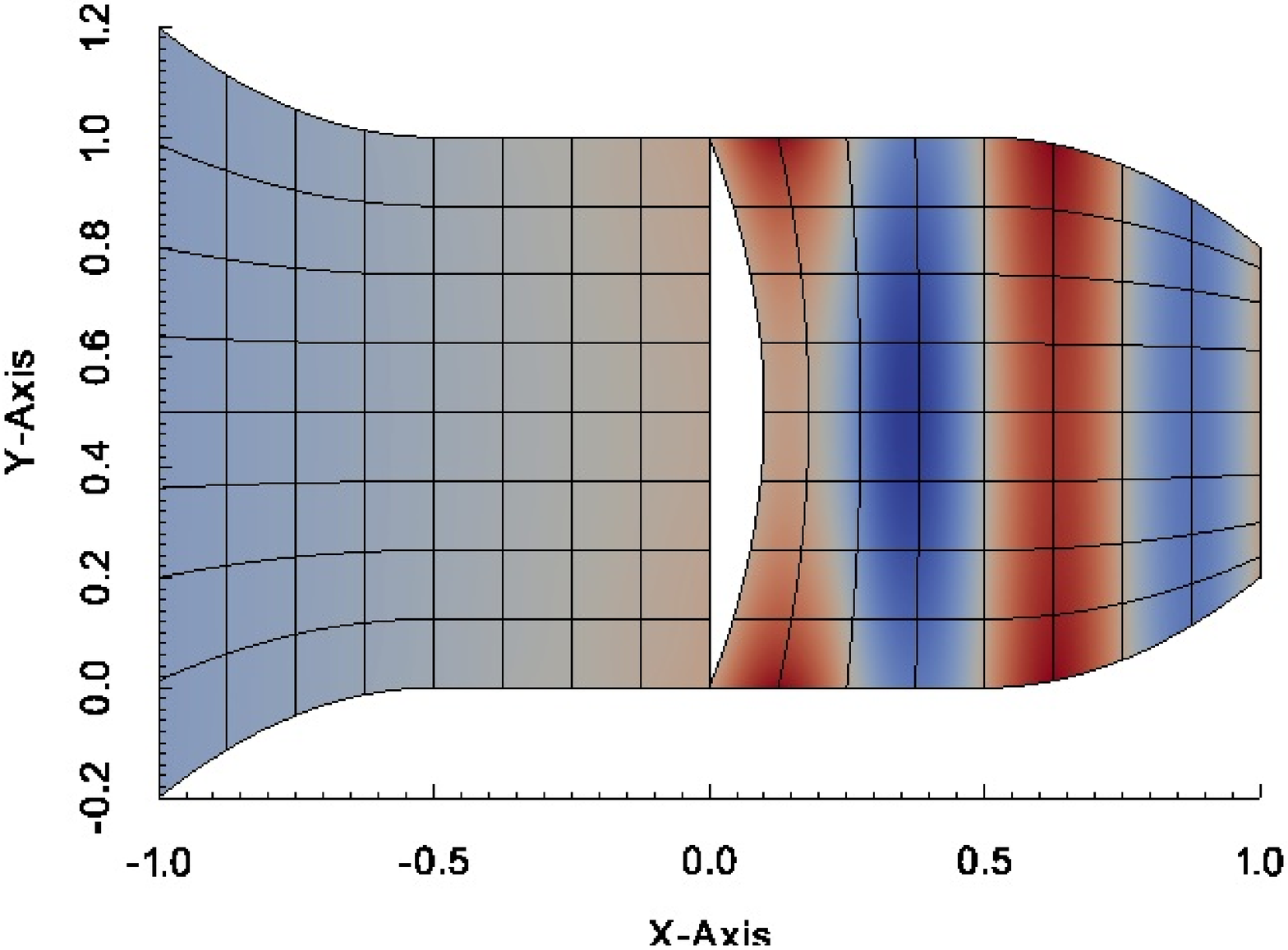}}
\subfigure[]{\includegraphics[width=4cm, height=4cm]{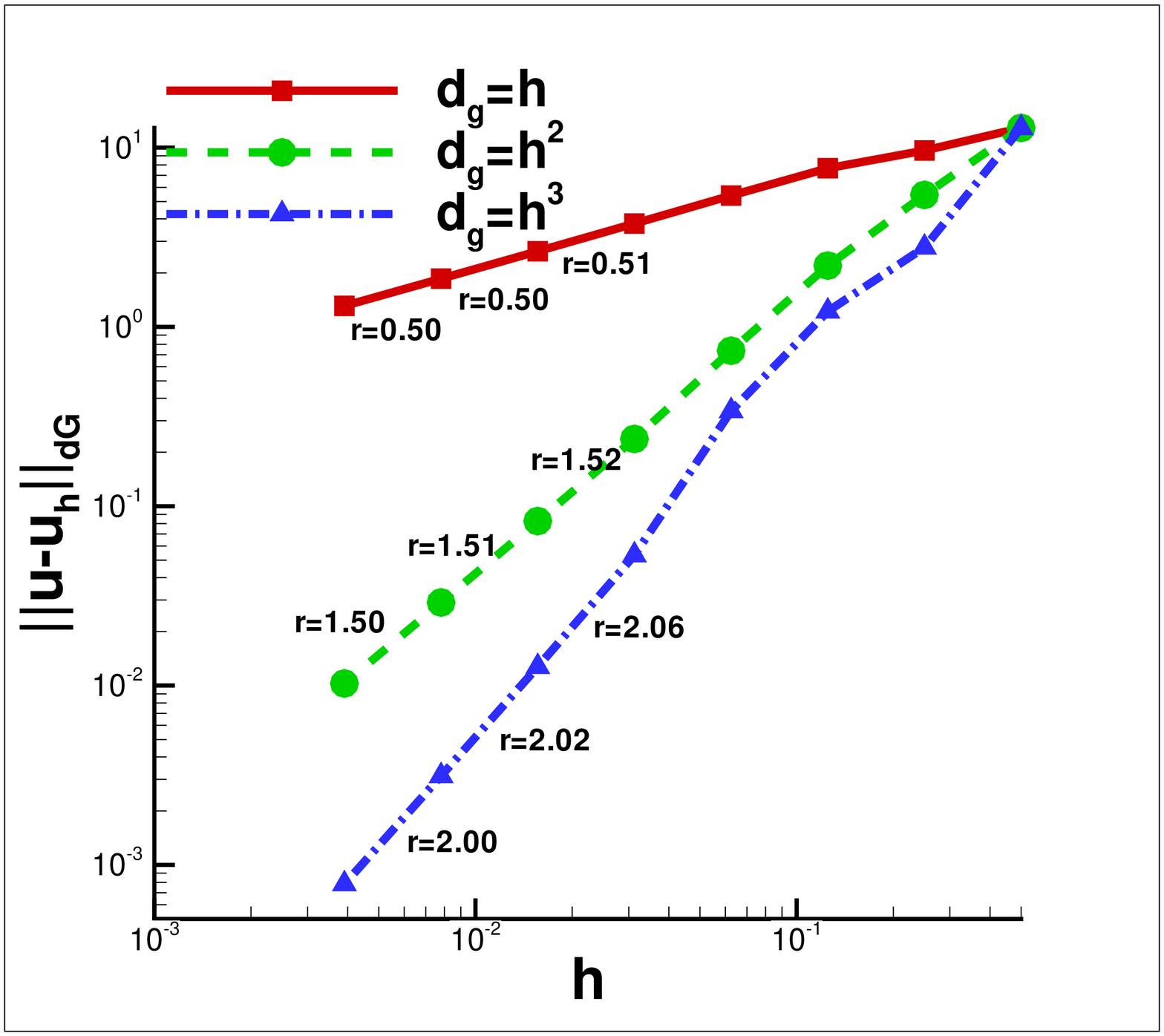}}
 \end{subfigmatrix}
  \caption{Example 2: (a) The contours of exact $u$  given by (\ref{NE_1}),
                      (b) The contours of $u_h$ on subdomains $\Omega_l,\,\Omega_r$ ,
                      (c) The convergence rates for the three choices of $\lambda$. }
  \label{Fig2_Test_2Gaps}
\end{figure}
 	 \paragraph{Example  3} This example consists of problem with low regularity solution, 
	i.e., 
	$u(x_1,x_2)=( (x_1-0)^2+(x_2-0.5)^2)^{\frac{\gamma}{2}}$, with $\gamma=1,\, 1.5 $ and $\, 2$, 
 	 see also \cite{LT:LangerToulopoulos:2014a} and \cite{HLT:LangerMantzaflarisMooreToulopoulos:2015a}. 
 	 The computational domain is the same as in the previous examples. 	  The source 
	  function $f$ and $u_D$ are  manufactured 
 	 by the exact solution. The diffusion coefficient     has been defined to be $\rho=1$ everywhere. 
 	 Fig. \ref{Fig3_Test_2Gaps_PntSingul}(a) shows the contours of $u_h$ using the values $\gamma=1$ and $d_g={0.1}$.
 	 By this example, we demonstrate the ability of the proposed
 	 method to approximate the solution singularities located on the interface of the subdomains with the expected accuracy. 
 	 We emphasize that  the convergence rate is specified by both, the regularity of the solution and the size
 	 of the gap. Table \ref{rate_gamma_lambda} displays the expected rates $r$ for 
 	  several values of the parameters $\gamma$ and $\lambda$. Hence, for validation, we have computed the convergence rates 
 	  of varying  size $d_g=h^\lambda$ and for several values of $\gamma$, which specifies the regularity of the solution, see Table  \ref{rate_gamma_lambda}. 
 	  In   Fig. \ref{Fig3_Test_2Gaps_PntSingul}(b), we plot the convergence rates for $\gamma=1$ and $\lambda=1,\,2$ and $\,3$. 
          We observe that the rates computed on the last level meshes for  $\lambda=2$ and $\lambda=3$	confirm
          the theoretically  predicted rates, see Table \ref{rate_gamma_lambda}. On the other hand, the rates corresponding to
          the case $\lambda=1$ are little higher than the expected. 
          We observe same behavior for  the rates plotted in Fig. \ref{Fig3_Test_2Gaps_PntSingul}(c) and in Fig. \ref{Fig3_Test_2Gaps_PntSingul}(d), which are related to the cases $\gamma=1.5$ and $\gamma=2$. The rates corresponding to $\lambda=2$ and $\lambda=3$ are approaching
          the optimal rates and are in agreement with the theoretical
          rates. But the rates for $\lambda=1$ are higher. 
           However, this result can be explained by the ``quadratic properties'' of the solution on $\partial \Omega_g$.
          More precisely, the expected rate $r=0.5$ is coming from the estimate of the first remainder term, see
          first term in the right hand side in (\ref{7_1c_b}) and (\ref{7_00_e}). Setting $d_g=h$ in
          the first term in the right hand side in (\ref{7_1c_b}), we get a ``continuous quadratic'' flux term,
          whose discrete  (second order)  analogue term appears in the dG numerical flux formula, see (\ref{7_d3_a}).
          In other words, the ``quadratic'' remainder flux term is implicitly approximated  through the numerical flux
          by second order B-spline space. Hence,  the resulting error of 
          the first remainder term, the bound of which is related to the second
          term in the formula of $\beta$ in  Theorem \ref{Theorem_1_estimates}, seems  to be very weak. 
          The rate that we found is approaching the value $r=1$ and is the expected value
          according to the first and third term of the formula of $\beta$, see Theorem \ref{Theorem_1_estimates}.
          \par
	  We also study the convergence rates for $\gamma=0.42$ and $\gamma=0.38$ 
	  which lead to solutions $u$ belonging to $W^{2,1.26}(\Omega)$
	  and $W^{2,1.23}(\Omega)$, respectively.
	  The convergence rates are plotted
          in Fig.  \ref{Fig3_Test_2Gaps_PntSingul}(e) and 
          in  in Fig.  \ref{Fig3_Test_2Gaps_PntSingul}(f), correspondingly. 
          Here, the suboptimal behavior of the rates can be seen for all $\lambda$ cases.
          Again for the case  $\lambda=1$,  the rates  are little higher than the expected ones, see first two rows in Table \ref{rate_gamma_lambda}. 
          Considering $\lambda=2$ and $\lambda =3$, the rates are determined by the regularity of the solution, because the approximation error is
          of lower order compared to the  estimates of the Taylor terms  $R_{\Omega_g}$, see the orders $\delta_\Pi(l,p,d)$ and 
          $\beta$ in (\ref{4.5_e}). The rates presented in the graphs in Fig. \ref{Fig3_Test_2Gaps_PntSingul}(e) and
          in Fig. \ref{Fig3_Test_2Gaps_PntSingul}(f), follow
          the error bound $\mathcal{O}(h^{\delta_\Pi})$ given in (\ref{4.5_e}) and in Table \ref{rate_gamma_lambda}. 
          
           \begin{figure}
  \begin{subfigmatrix}{3}
 \subfigure[]{\includegraphics[scale=0.6155]{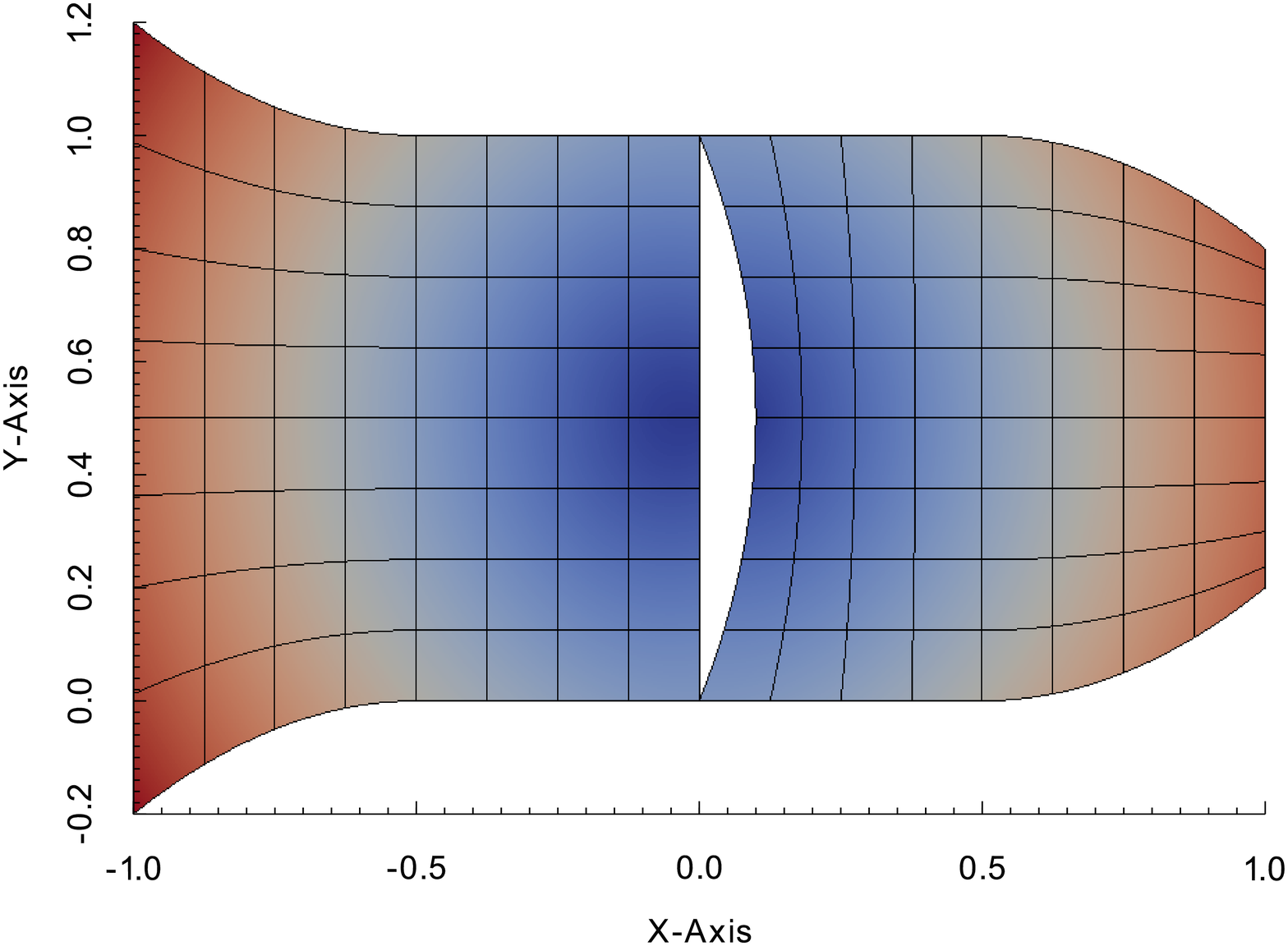}}
  \subfigure[$\gamma=1$]{\includegraphics[scale=0.10825]{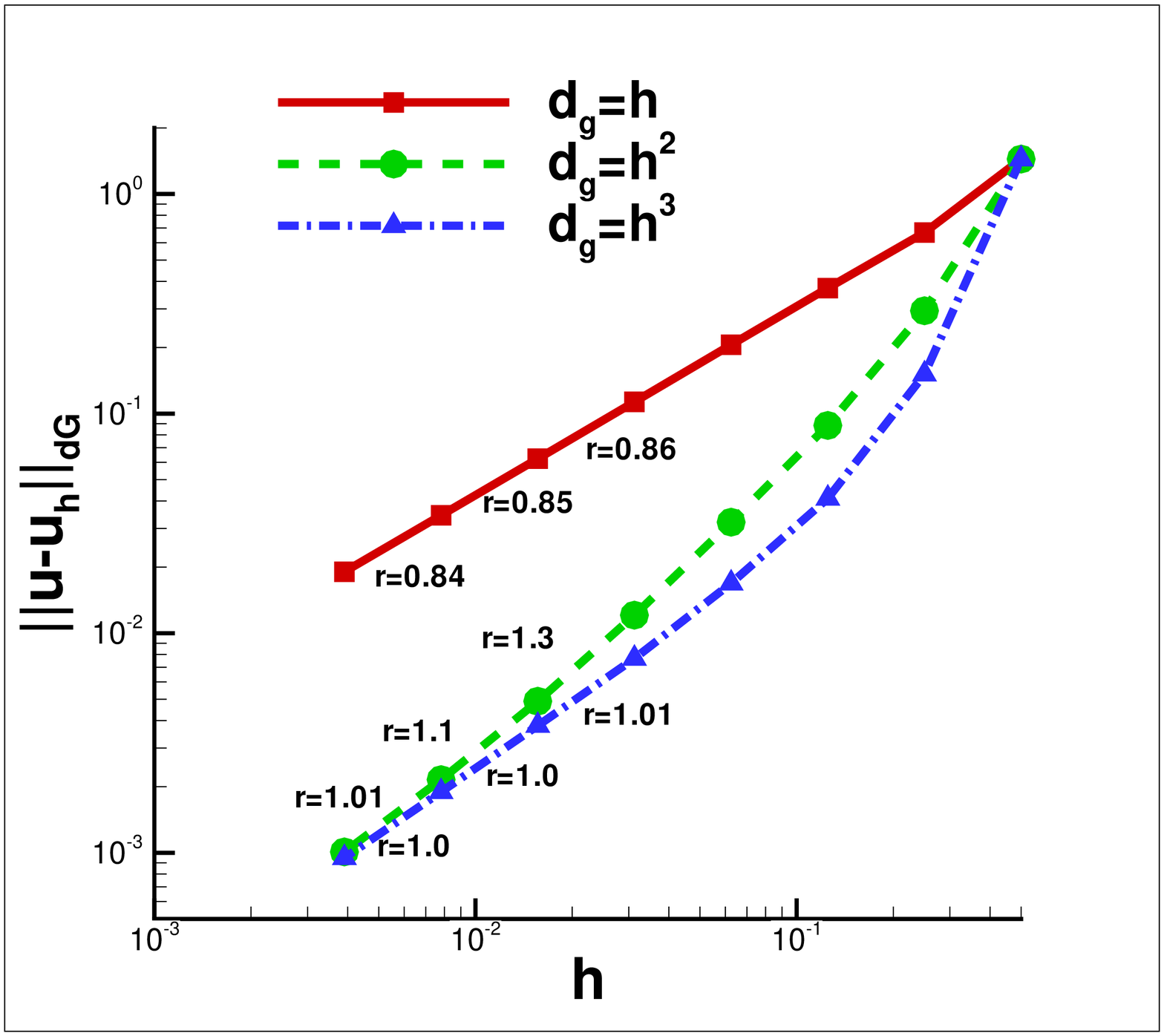}}
  \subfigure[$\gamma=1.5$]{\includegraphics[scale=0.10825]{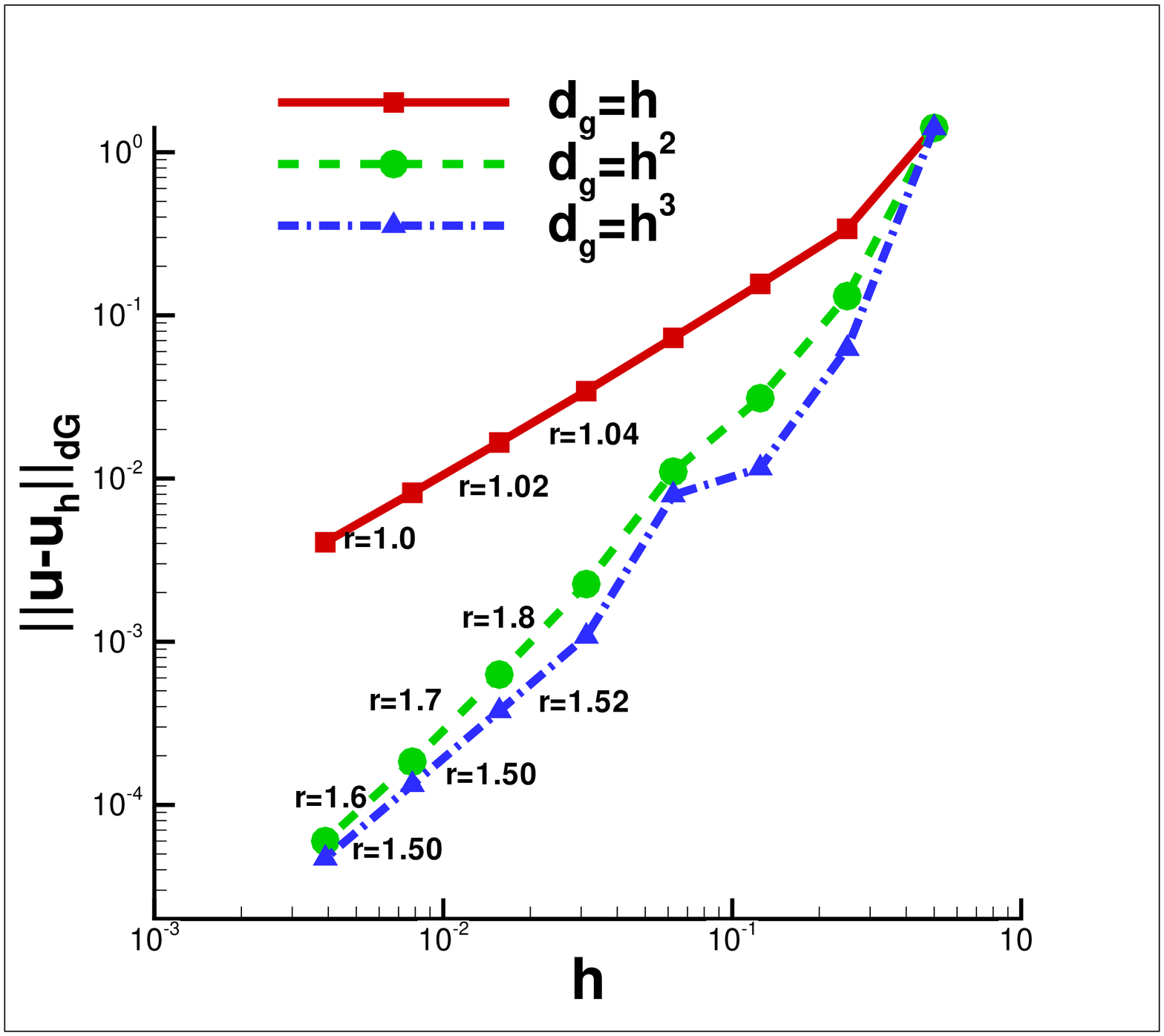}}
 \subfigure[$\gamma=2$]{\includegraphics[scale=0.10825]{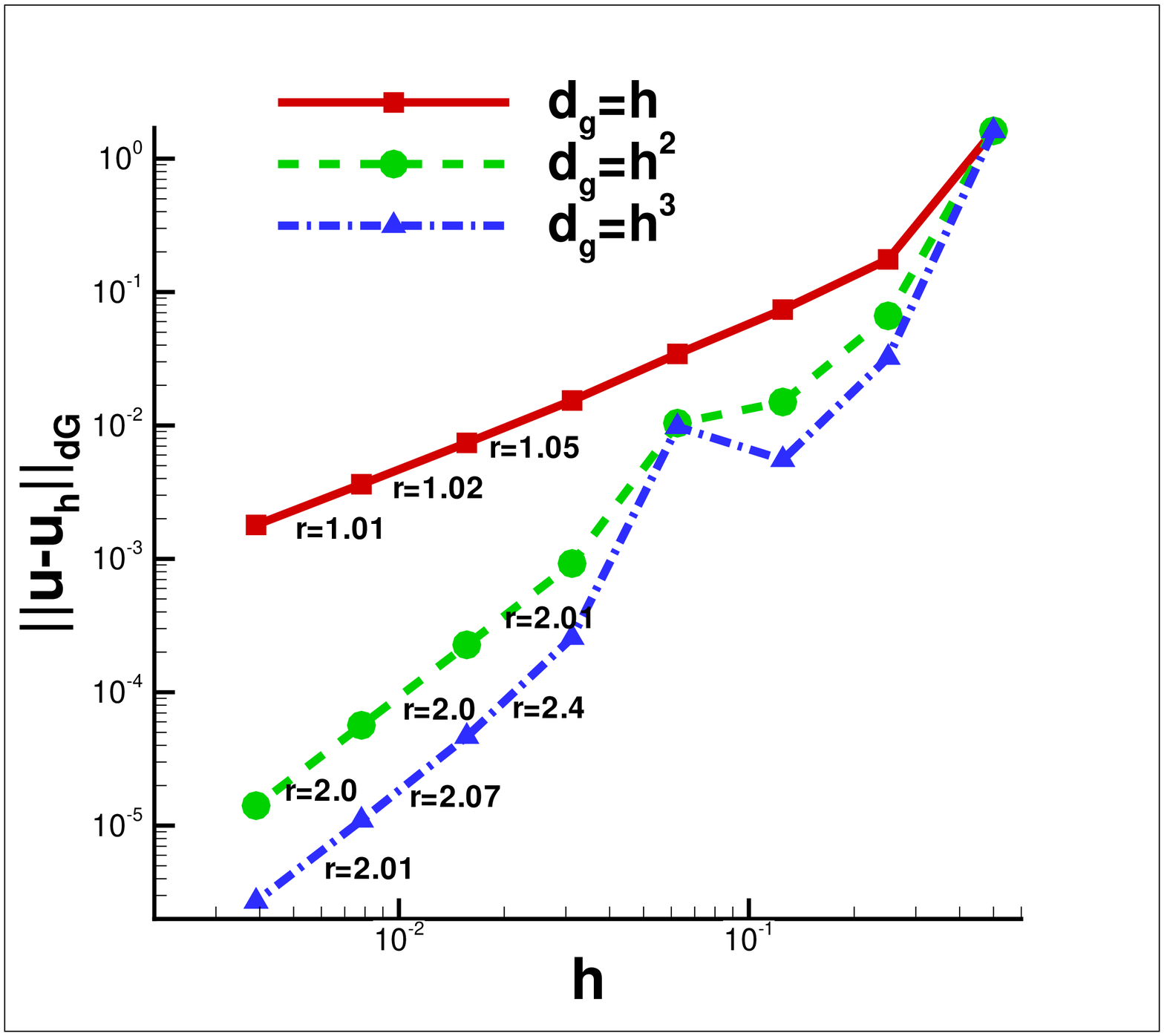}}
 \subfigure[$\gamma=0.42$]{\includegraphics[scale=0.10825]{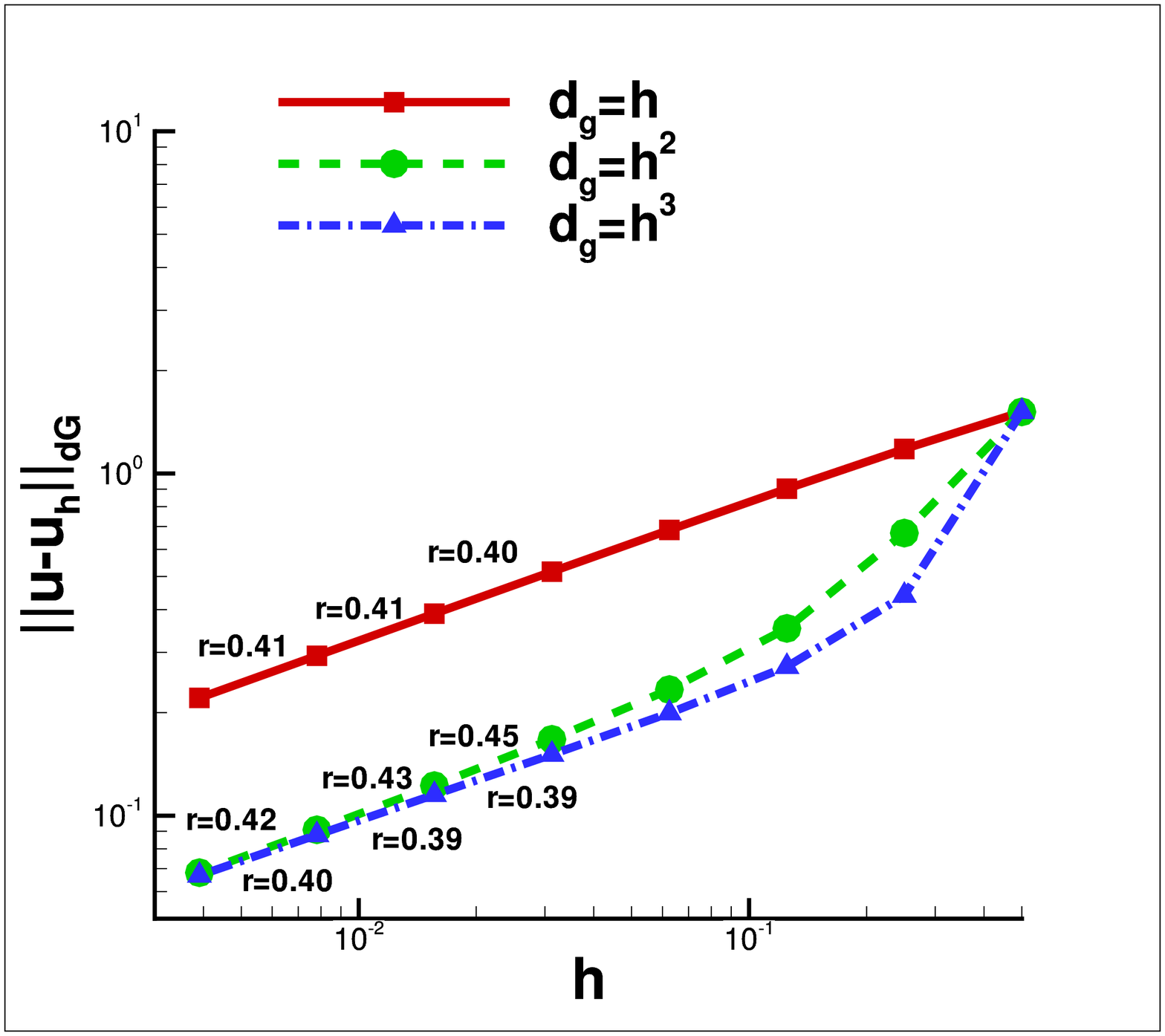}}
  \subfigure[$\gamma=0.38$]{\includegraphics[scale=0.10825]{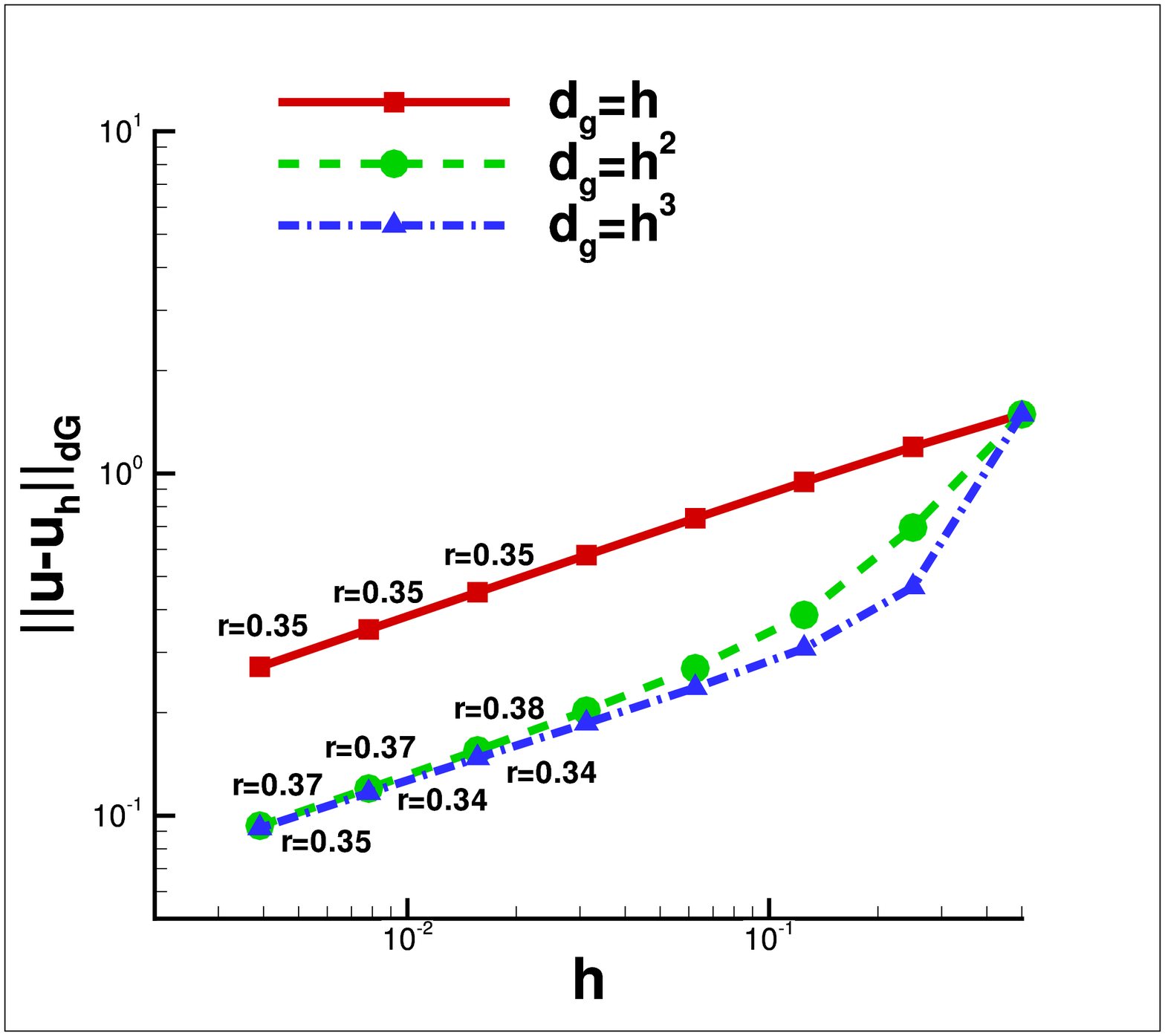}}
 \end{subfigmatrix}
  \caption{Example 3: (a) The contours of $u_h$ computed setting $\gamma=1$ and $d_g={0.1}$,
                      (b) convergence rates $r$ for $\gamma=1$, 
                      (c) Convergence rates $r$ for $\gamma=1.5$,
                      (d) Convergence rates $r$ for $\gamma=2$,
                      (e) Convergence rates $r$ for $\gamma=0.42$,
                      (f) Convergence rates $r$ for $\gamma=0.38$.}
  \label{Fig3_Test_2Gaps_PntSingul}
\end{figure}
 \subsection{Three-dimensional numerical examples}
 In the three-dimensional tests, the domain $\Omega$ has been constructed by  a straight  prolongation to the $x_3$-direction
 of the previous two-dimensional domain. The knot vector in $x_3$-direction  is the same as in the other directions,  
 this means $\Xi_i^3=\{0 , 0 ,0 ,0.5, 0.5 ,1, 1, 1\}$ with $i=l,r$.  The B-spline parametrizations of the two subdomains have been build by adding   a third component to the control points that are listed in Table  \ref{contrl_point_omega_lr}. The third  component takes the following values $\{0, 0.25, 0.50, 0.75, 1\}$. Again, the gap region is artificially constructed by moving only the interior control points located at the $x_2x_3$-plane into the $x_1$-direction.  
 %
  	\paragraph{Example  4} Although the first 3d example   is a simple extension of   the previous two dimensional   Example 2, it is still  interesting 
	to check  the numerical rates. The exact solution is given by (\ref{NE_1})  and the set up of the problem is illustrated in   	  Fig. \ref{Fig4_3DTest_1}. 
  	The interface $F$ is the $x_2x_3$-plane. In Fig. \ref{Fig4_3DTest_1}(a), we can see the contours of the solution $u_h$ on both
  	  subdomains without having a gap region. Note that the contours resemble  the two-dimensional contours along any slice
  	  $x_3=constant$. In Fig. \ref{Fig4_3DTest_1}(b), we plot the contours of the solution $u_h$ resulting from
  	  the solution of the problem  in case of having a gap region
  	  with $d_g=0.1$.  
  	  We can clearly observe the similarities of the contours in Fig. \ref{Fig2_Test_2Gaps}(b) and
  	  in  Fig. \ref{Fig4_3DTest_1}(b). 
  	  Also, in Fig. \ref{Fig4_3DTest_1}(b), we show the shape of the gap 
  	  as it appears on  an oblique cut of the domain $\Omega$. We have computed the convergence rates
  	  for the three different values of $\lambda$. The results of the computed rates are plotted in Fig. \ref{Fig4_3DTest_1}(c). We observe that all the rates are in agreement with the predicted rates by the theory and are
  	  similar to the rates of the two-dimensional test Example 2, see Fig. \ref{Fig2_Test_2Gaps}(c).

 \begin{figure}
  \begin{subfigmatrix}{3}
\subfigure[]{\includegraphics[scale=0.165]{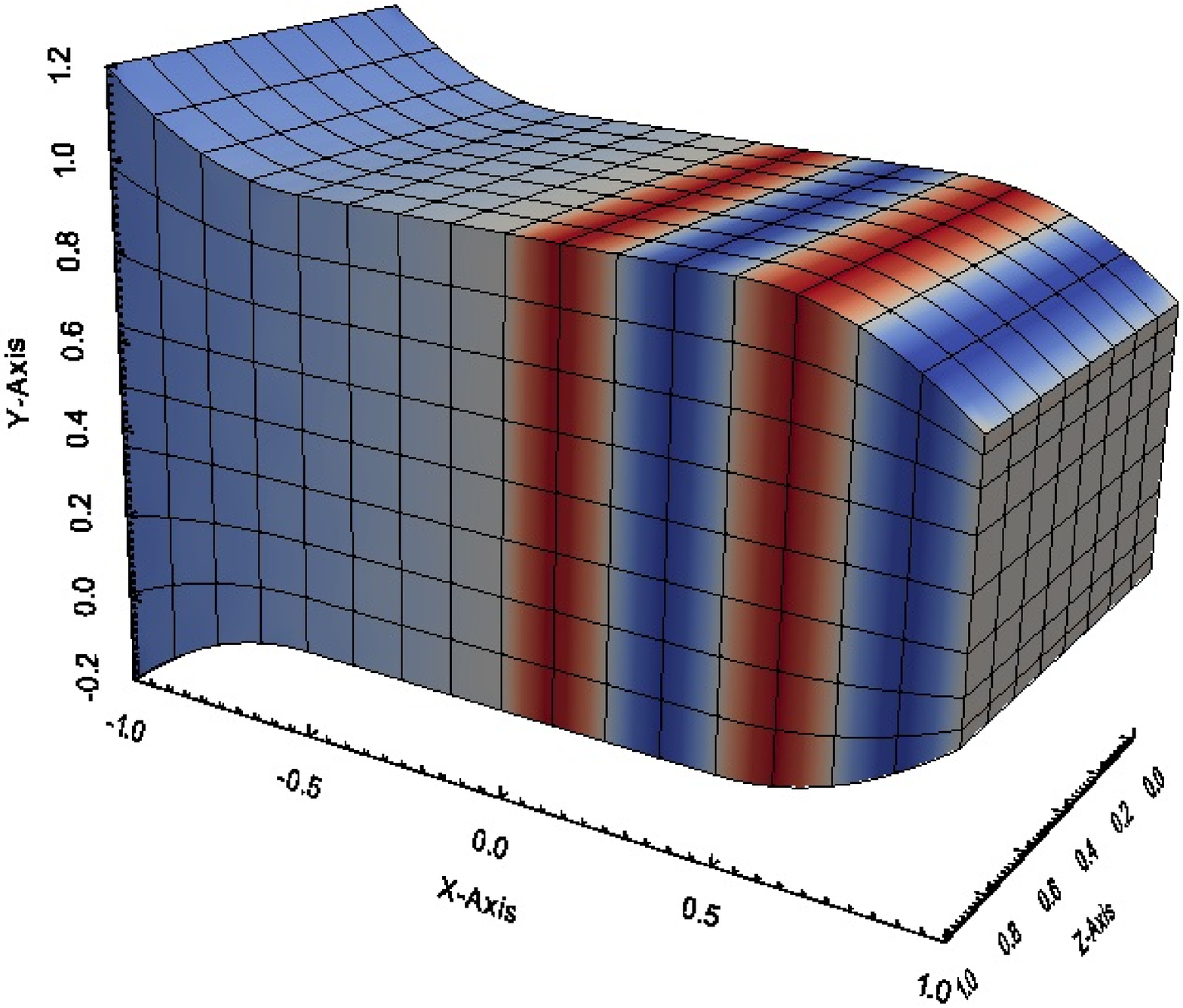}}
 \subfigure[]{\includegraphics[scale=0.165]{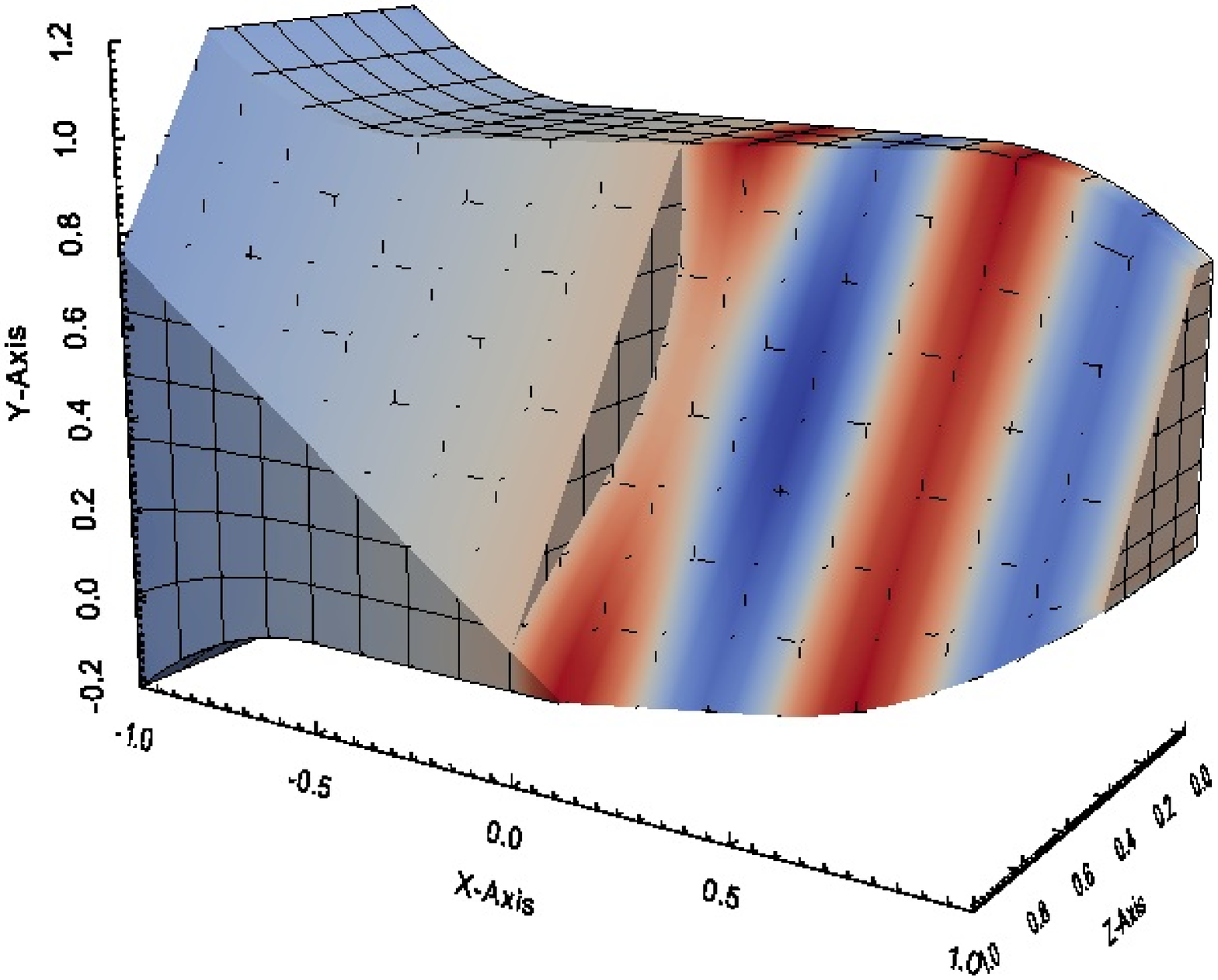}}
  \subfigure[]{\includegraphics[scale=0.155]{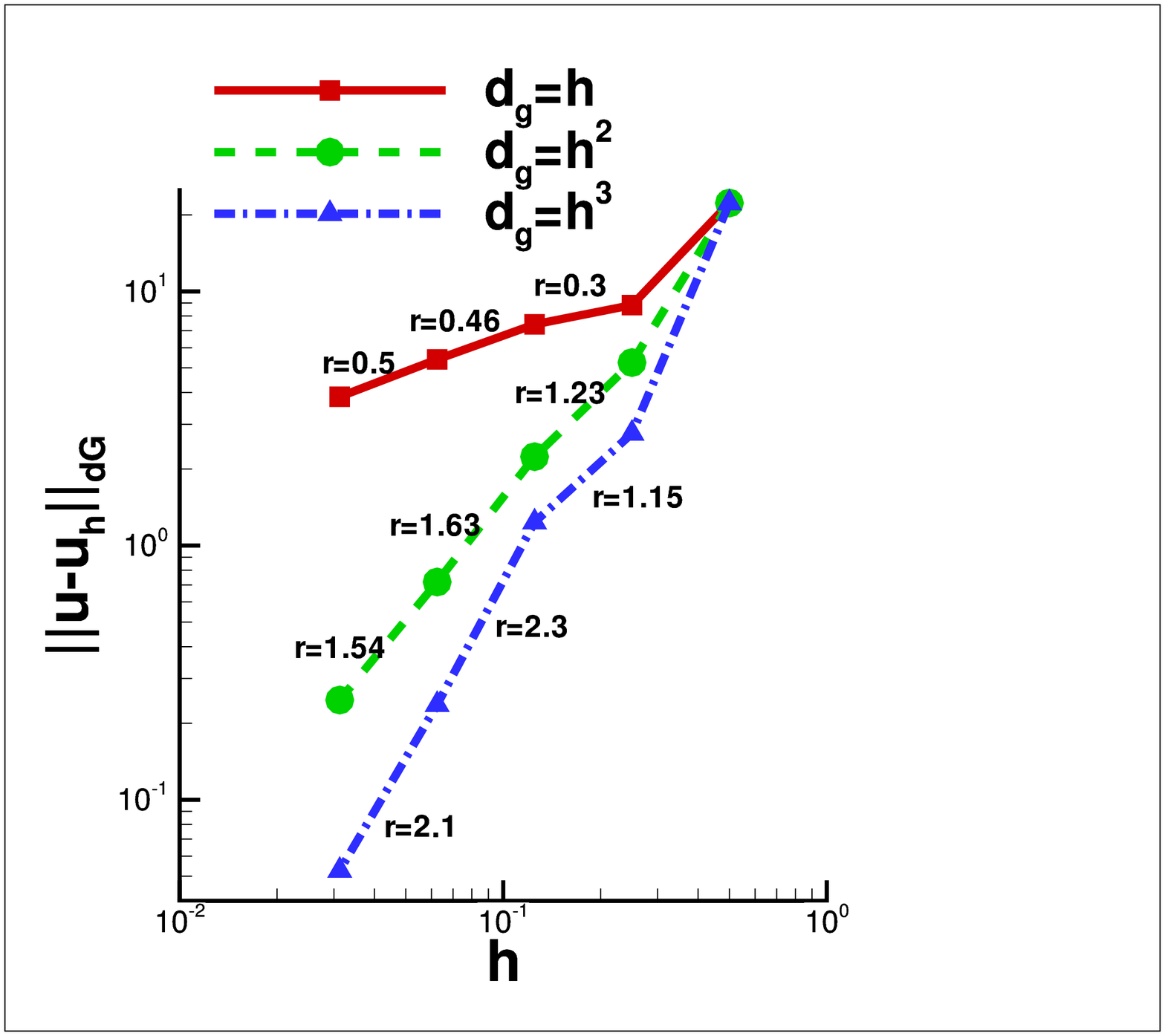}}
 \end{subfigmatrix}
  \caption{Example 4: (a) The contours of $u_h$ computed on $\Omega$ ,
                      (b) The contours of $u_h$ computed on $\Omega\setminus\overline{\Omega}_g$ with $d_g=0.1$ , 
                      (c) Convergence rates $r$.}
  \label{Fig4_3DTest_1}
\end{figure}

  \paragraph{Example  5, \textit{$\Phi$-shape gap}}For the second numerical test example, 
the domain $\Omega$ is the same as in previous example with the
                    subdomain interface $F$ to be the $x_2x_3$-plane,
                    see Figs. \ref{Fig4_3DTest_1}(a) and \ref{Fig4_3DTest_2}(a).  
		    We consider a manufactured 
		    problem, where the solution is
              \begin{align}\label{3D_test_2}
                  	u(x_1,x_2,x_3)=
  	                  \begin{cases}
                             u_1:=\sin(\pi (x_1+x_2))      & \text{if } (x_1,x_2,x_3) \in \Omega_1, \\
                             u_2:= \sin(\pi(4 x_1+x_2)) & \text{if} {\ }(x_1,x_2,x_3)\in \Omega_2,
                           \end{cases}
                  \end{align}
                  and the diffusion coefficient is defined to be  $\rho_1=4$ and $\rho_2=1$.
             We note that, for this test case, 
	      we have
	      $\llbracket u \rrbracket |_{F}=\llbracket \rho \nabla u\rrbracket |_{F}=0$. \\
                          Here, we  artificially created the gap region such that $\Omega_g \cap \Omega_1 \neq \emptyset$, 
             and consequently for the left gap boundary $F_l$, we get
	      $F_l\neq F$.  
             In particular, in all computations for this example, 
             the two gap parts belonging to $\Omega_1$ and $\Omega_2$ are symmetric with respect to the original interface $F$. In Fig. \ref{Fig4_3DTest_2}(b),  we can see the gap shape for $d_g=2\cdot 0.0625$. 
             The contours of the solution $u_h$ computed on 
	      a decomposition 
	      with $d_g=0$ and $d_g=2\cdot 0.0625$ are presented
             in Fig. \ref{Fig4_3DTest_2}(a) and in Fig. \ref{Fig4_3DTest_2}(b), respectively. 
	     We have computed the convergence rates $r$
             for the three different sizes $d_g$ , i.e., $d_g=h^\lambda$ with $\lambda=1\,,\lambda=2$ and $\lambda=3$. We plot our results   in Fig. \ref{Fig4_3DTest_2}(c). We observe that the rates are approaching the expected values that have been
             mentioned in Table \ref{table_value_r}. Furthermore, we note that, for the case  $d_g=h$, the rate tends to become
              0.5 and is in agreement with the rate predicted  by the theory, see Table \ref{table_value_r} and Theorem \ref{Theorem_1_estimates}. Therefore, in this example, the rates  follow the same behavior
             as in the previous 3d example and do not follow the behavior of the two-dimensional low
             regularity test case, see Fig. \ref{Fig3_Test_2Gaps_PntSingul}.

 \begin{figure}
  \begin{subfigmatrix}{3}
\subfigure[]{\includegraphics[scale=0.165]{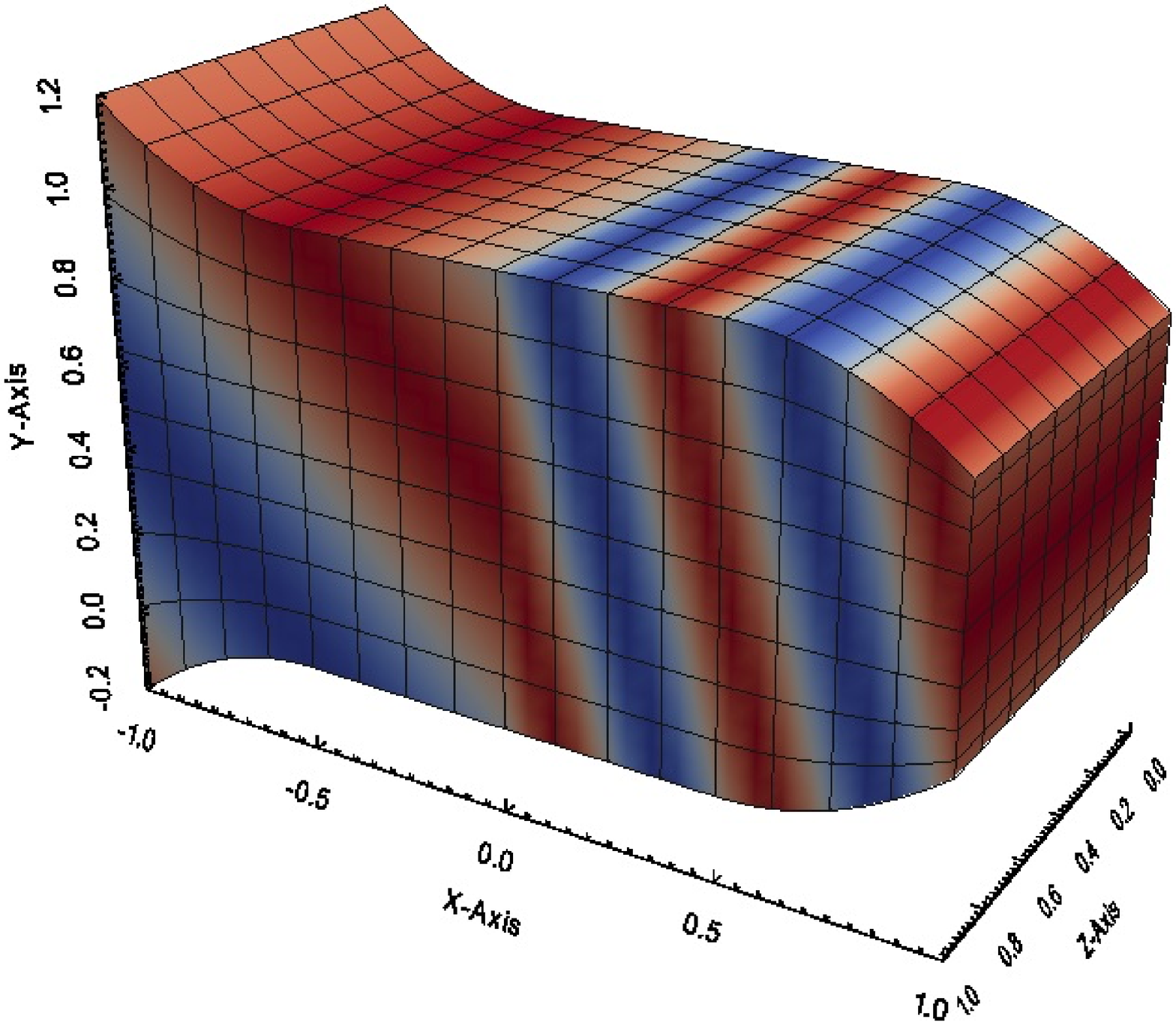}}
 \subfigure[]{\includegraphics[scale=0.15]{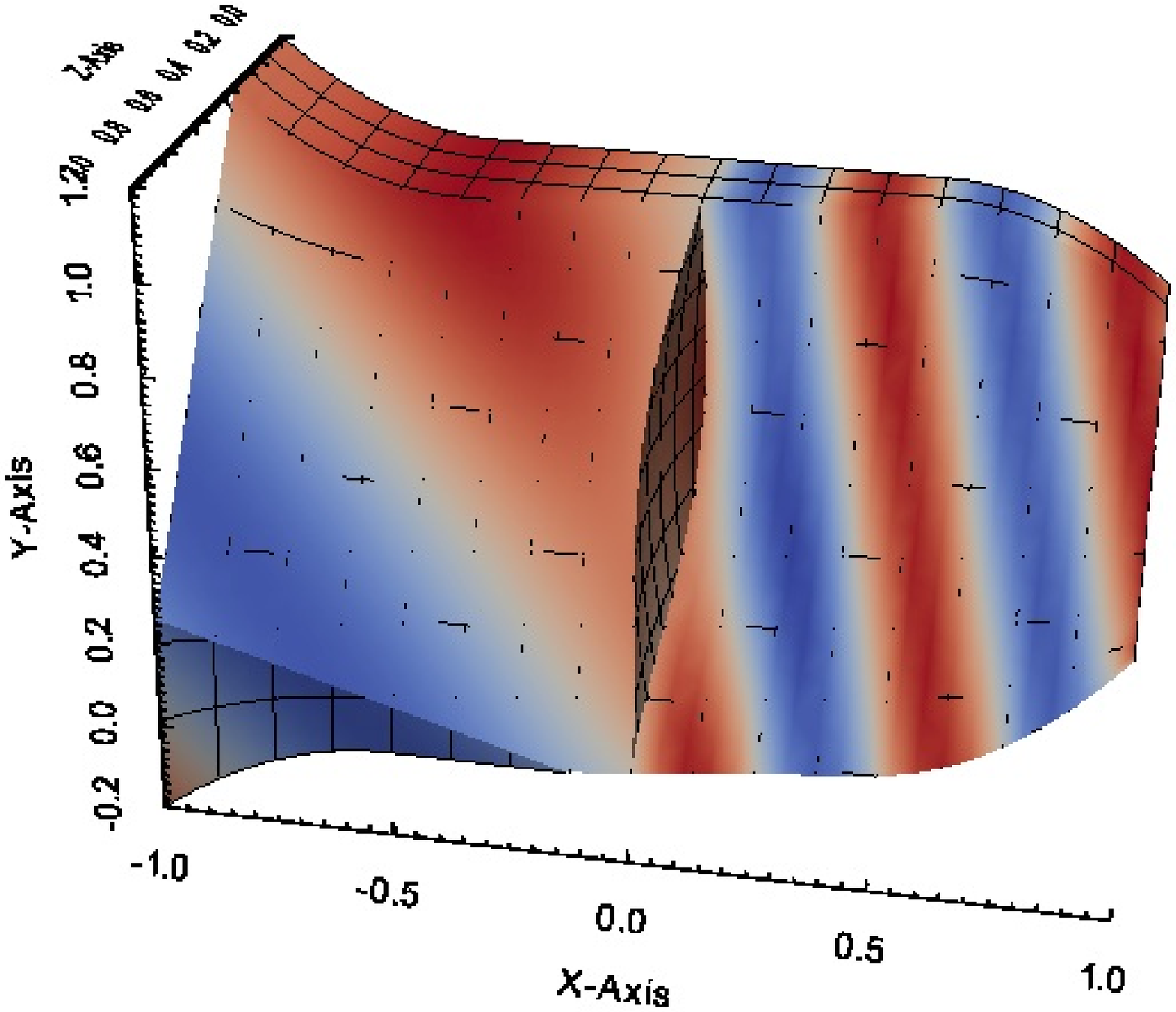}}
  \subfigure[]{\includegraphics[scale=0.155]{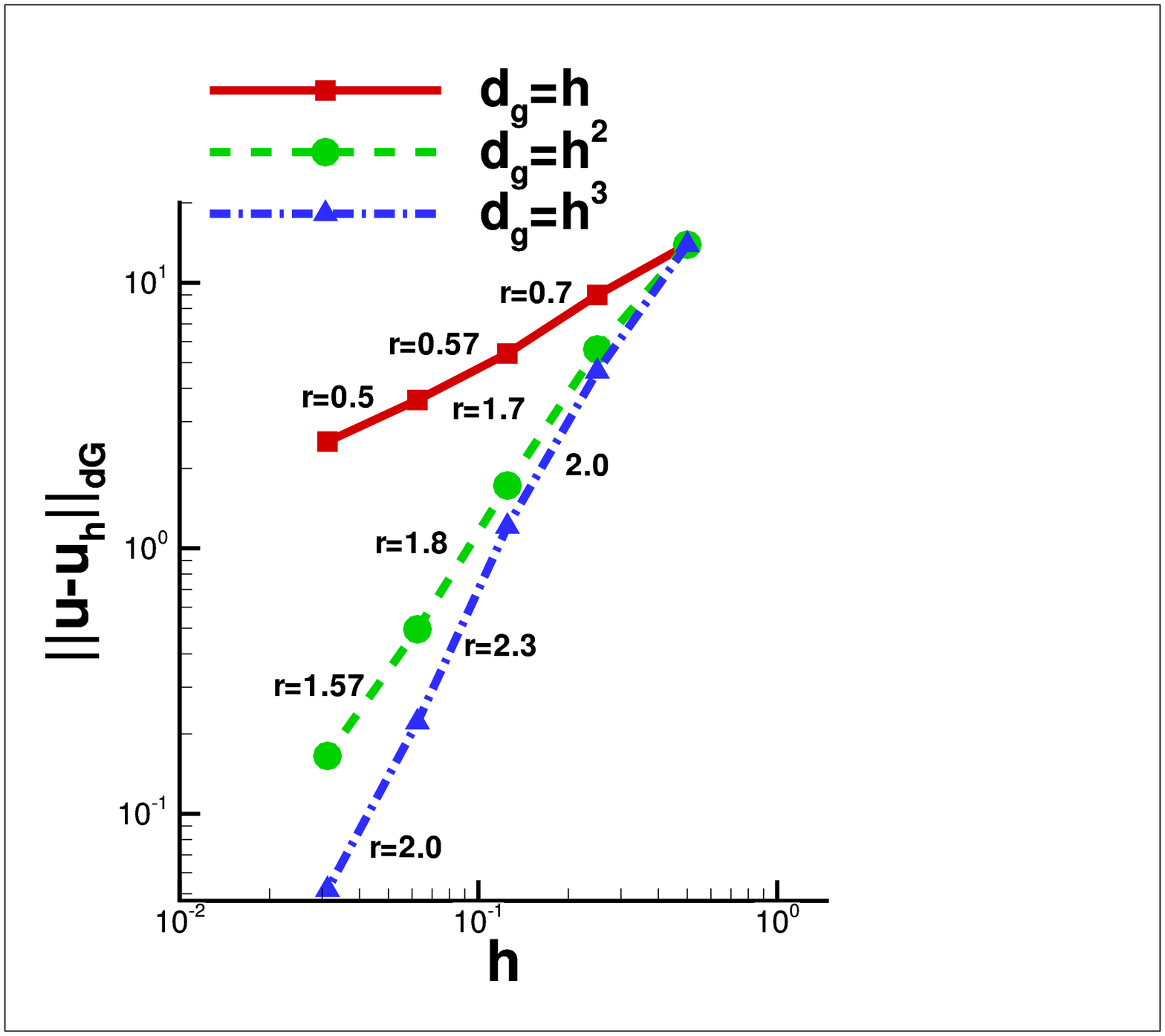}}
 \end{subfigmatrix}
  \caption{Example 5: (a) The contours of $u_h$ computed on $\Omega$ ,
                      (b) The contours of $u_h$ computed on $\Omega\setminus\overline{\Omega}_g$ with $d_g=2\cdot0.0625$ , 
                      (c) Convergence rates $r$ for the different $d_g$ sizes.}
  \label{Fig4_3DTest_2}
\end{figure}     

\section{Conclusion and outlook}
In this article,  we have developed and analyzed dG IgA methods for discretizing linear,
second-order elliptic boundary value problems 
on volumetric patch decompositions
with non-matching interface parametrizations, which include gap regions between the 
adjacent subdomains (patches). Starting from the original weak 
formulation,
we derived a consistent variational problem
on a  decomposition without including the gap region. 
The unknown normal fluxes on the gap boundary were approximated
via Taylor expansions.
These approximations were adapted to the proposed dG IgA scheme, and the communication
of the discrete solution of the  adjacent subdomains was ensured. A priori error estimates in the 
dG-norm $\|.\|_{dG}$
were shown in terms of 
the mesh-size $h$ and the gap distance $d_g$. 
The estimates 
were
confirmed
by solving several two- and three-dimensional test problems with known exact solutions. 
\par
The 
techniques
presented here for linking the diametrically opposite points on the gap boundary can be used 
for a wide variety of other problems with different gap shapes. 
The only information, which is required,
is the construction of parametrization between the opposite  gap boundary parts. 
This parametrization can be used in conjunction with the Taylor expansions to derive
approximations of the normal fluxes on the gap boundary and then to incorporate the numerical fluxes into the 
dG IgA scheme. 
From a practical point of view, it would be valuable to derive  a posteriori error estimates with computable upper bounds, see, e.g.,
\cite{HLT:Verfuerth:2013a,HLT:Repin:2008a,HLT:KleissTomar:2015a}.
Fast generation techniques for the IgA system matrix and fast parallel solvers for large-scale systems of dG IgA equations
are certainly other hot research topics. Fast generation techniques can be 
developed  
on the basis of low-rank tensor approximations 
as proposed in \cite{HTL:MantzaflarisJuettlerKhoromskijLanger:2015a}.
Efficient solvers can certainly be constructed on the basis of multigrid, multilevel,
and domain decomposition methods.
In particular, IETI-DP methods,
introduced in  
\cite{HLT:KleissPechsteinJuettlerTomar:2012a}  
and analysed in 
\cite{HLT:HoferLanger:2015a}, 
see also 
\cite{HLT:BeiraoChoPavarinoScacchi:2013a,HLT:BeiraoPavarinoScacchiWidlundZampini:2014a} 
for related BDDC methods, 
seem to be well suited for the parallel solution 
of dG IgA equations including the dG IgA schemes studied in this paper.
\section*{Acknowledgments}
This work was supported by the Austrian Science Fund (FWF) under the grant NFN S117-03. 
This support is gratefully acknowledged. 
\bibliographystyle{plain} 
\bibliography{dGIgAgaps}

\end{document}